\documentclass[11pt,oneside]{article}

\usepackage{latexsym}
\usepackage{shortvrb}
\usepackage{graphicx}
\usepackage{fancyhdr}
\usepackage{amsfonts}
\usepackage{amsmath}
\usepackage{amssymb}
\usepackage{mathrsfs}
\usepackage{makeidx}
\usepackage[all]{xy}
\usepackage{amsthm}
\usepackage{tikz}
\usepackage{MnSymbol}
\usepackage{cite}

\usepackage{hyperref}

\hypersetup{
    colorlinks=true,
    linkcolor=blue,
    filecolor=magenta,      
    urlcolor=cyan,
}

\usetikzlibrary{matrix}

\usepackage[T1]{fontenc}
\usepackage{amsfonts}
\usepackage{amsmath}

\newtheorem{thm}{Theorem}[section]
\newtheorem{lem}[thm]{Lemma}
\newtheorem{prop}[thm]{Proposition}
\newtheorem{ex}[thm]{Example}
\newtheorem{cor}[thm]{Corollary}

\theoremstyle{definition}
\newtheorem{definition}[thm]{Definition}
\newtheorem{obs}[thm]{Observation}
\newtheorem{remark}[thm]{Remark}

\newtheorem{prob}[thm]{Problem}

\newcommand{\blackged}{\hfill$\blacksquare$}
\newcommand{\whiteged}{\hfill$\square$}
\newcounter{proofcount}

\renewenvironment{proof}[1][\proofname.]{\par
  \ifnum \theproofcount>0 \pushQED{\whiteged} \else \pushQED{\blackged} \fi%
  \refstepcounter{proofcount}
  \normalfont 
  \trivlist
  \item[\hskip\labelsep
        \itshape
    {\bf\em #1}]\ignorespaces
}{%
  \addtocounter{proofcount}{-1}
  \popQED\endtrivlist
}

\begin{document}

\begin{center}
    \textbf{Lifting bicategories through the Grothendieck construction}
    
    \
    
    \textit{Juan Orendain}

\end{center}

\

\noindent \textit{Abstract:} We treat the problem of lifting bicategories into double categories through categories of vertical morphisms. We make use of a specific instance of the Grothendieck construction to provide, for every bicategory equipped with a possible vertical category, together with a suitable monoidal pre-cosheaf relating these two structures, a double category lifting the decorated bicategory along the category of vertical morphisms provided as set of initial conditions. We prove in particular that every decorated bicategory admits a lift to a double category. We study relations of instances of our construction to foldings, cofoldings, framed bicategories and globularily generated double categories.

\tableofcontents

\section{Introduction}

\noindent Double categories were introduced by Ehresmann in \cite{Ehr1,Ehr2}. 2-categories and bicategories were later introduced by Benabeu in \cite{Benabou2cats} and \cite{BenabouBicats} respectively. Both concepts model 2-dimensional categorical structures, each with its advantages and disadvantages. Double categories accomodate more flexible structures, see \cite{SegalCFT,SchulmanDerived} while bicategories are better behaved with respect to pasting, see \cite{Power} as opposed to \cite{DawsonPare}, for example. There are many classical ways of relating double categories and bicategories. Every double category has an underlying horizontal bicategory and every bicategory can be considered as a trivial double category. More interestingly the Ehresmann double category of quintets construction \cite{Ehr4} and the category of adjoints construction \cite{Palmquist} non-trivially associate double categories to bicategories.   

We consider the following problem: Given a bicategory $\mathcal{B}$ and a category $\mathcal{B}^*$ such that the collection of 0-cells of $\mathcal{B}$ and the collection of objects of $\mathcal{B}^*$ are equal, we wish to construct interesting double categories $C$ having $\mathcal{B}$ as horizontal bicategory and having $\mathcal{B}^*$ as category of objects. We express this through the equation $H^*C=(\mathcal{B}^*,\mathcal{B})$. In that case we call $C$ an internalization of the pair $(\mathcal{B}^*,\mathcal{B})$. Put pictorially, if we are given a collection of diagrams of the form:

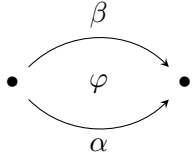
\begin{center}
\begin{tikzpicture}
  \matrix (m) [matrix of math nodes,row sep=4em,column sep=4em,minimum width=2em]
  {\bullet&\bullet\\};
  \path[-stealth]
    (m-1-1) edge [bend right=45] node [below]{$\alpha$}(m-1-2)
            edge [bend left=45] node [above]{$\beta$} (m-1-2)
            edge [white] node [black][fill=white] {$\varphi$} (m-1-2);
\end{tikzpicture}
\end{center}

\noindent forming a bicategory $\mathcal{B}$ and a collection of vertical arrows:

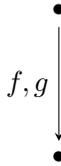
\begin{center}
\begin{tikzpicture}
  \matrix (m) [matrix of math nodes,row sep=4em,column sep=4em,minimum width=2em]
  {\bullet\\
  \bullet\\};
  \path[-stealth]
    (m-1-1) edge node [left]{$f,g$}(m-2-1);
\end{tikzpicture}
\end{center}

\noindent forming a category $\mathcal{B}^*$, such that these diagrams are related by the fact that the 0-cells above and below are the same, we wish to understand ways of coherently combining these structures into collections of squares:

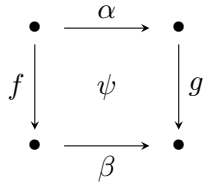
\begin{center}
\begin{tikzpicture}
  \matrix (m) [matrix of math nodes,row sep=3em,column sep=3em,minimum width=2em]
  {
     \bullet&\bullet \\
     \bullet&\bullet \\};
  \path[-stealth]
    (m-1-1) edge node [above]{$\alpha$}(m-1-2)
            edge [white] node [black][fill=white]{$\psi$}(m-2-2)    
    (m-2-1) edge node [below] {$\beta$} (m-2-2)
    (m-1-1) edge node [left]{$f$} (m-2-1)
    (m-1-2) edge node [right]{$g$} (m-2-2);
\end{tikzpicture}

\end{center}

\noindent forming double categories $C$, such that the arrows of the left and right edges of the the above squares are precisely the arrows of $\mathcal{B}^*$ and such that the squares of the form:

\begin{center}
\begin{tikzpicture}
  \matrix (m) [matrix of math nodes,row sep=3em,column sep=3em,minimum width=2em]
  {
     \bullet&\bullet \\
     \bullet&\bullet \\};
  \path[-stealth]
    (m-1-1) edge node [above]{$\alpha$}(m-1-2)
            edge [white] node [black][fill=white]{$\psi$}(m-2-2)    
    (m-2-1) edge node [below] {$\beta$} (m-2-2)
    (m-1-1) edge [blue]node [black][left]{$id_\bullet$} (m-2-1)
    (m-1-2) edge [blue]node [black][right]{$id_\bullet$} (m-2-2);
\end{tikzpicture}

\end{center}

\noindent can be identified with the diagrams defining $\mathcal{B}$. Solutions to this problem are easily seen to exist in specific cases, e.g. the case in which $\mathcal{\mathcal{B}}^*$ is trivial is solved by the trivial double category construction and the case in which $\mathcal{B}^*$ and the horizontal category of $\mathcal{B}$ are equal is solved by the Ehresmann double category of quintets construction for $\mathcal{B}$. An interesting instance of this problem is the case in which $\mathcal{B}^*$ is the category of von Neumann algebras and their morphisms and $\mathcal{B}$ is the bicategory of Hilbert bimodules, see \cite{Bartels1,Bartels2,Landsman}. We provide solutions to this problem for general $\mathcal{B}^*$ and $\mathcal{B}$.

The Grothendieck construction establishes an equivalence between the categorical notion of Grothendieck fibration and the algebraic notion of presheaf on \textbf{Cat}. Intuitively Grothendieck fibrations are categorical structures capturing the notion of category parametrized by categories. The Grothendieck construction puts this notion into algebraic terms. We make use of a specific instance of the Grothendieck construction to provide solutions to the problem presented above. Given a bicategory $\mathcal{B}$ and a category $\mathcal{B}^*$ as above, we construct, for every pre-cosheaf $\Phi$ of $\mathcal{B}^*$ on \textbf{Cat} satisfying certain conditions related to $\mathcal{B}$, a double category $C^\Phi$ solving the equation $H^*C^\Phi=(\mathcal{B}^*,\mathcal{B})$ posed above. In particular we prove that such equation admits solutions for every pair $(\mathcal{B}^*,\mathcal{B})$. The Grothendieck construction has already been studied in the context of limits in double category in \cite{GrandisPare}.

Recent techniques in the theory of double categories use additional structure to reduce questions regarding the theory of double categories to questions of associated bicategories. These techniques include the theory of holonomies of Brown and Spencer \cite{BrownSpencer}, the theory foldings and cofoldings of Brown and Mosa \cite{BrownMosa} and the theory of framed bicategories of Schulman \cite{SchulmanFramed} among others. These structures relate vertical and horizontal structures of double categories and establish correspondences between certain squares on double categories with simpler squares in a coherent way. We study relations of double categories constructed through our methods and the conditions mentioned above. Our techniques are closely related to the techniques introduced by Schulman in his treatise of monoidal fibrations and framed bicategories. We prove that not all double categories constructed through our methods are framed bicategories. The exact relation between double categories constructed through our methods and framed bicategories constructed via monoidal fibrations is yet to be explored.

The problem of finding solutions to the equation $H^*C=(\mathcal{B}^*,\mathcal{B})$ has already been studied by the author in \cite{yo1,yo2} where the notion of globularily generated double category, or GG double category for short, and vertical length were presented. GG double categories are minimal solutions to the above problem and we thus study relations of this condition and double categories constructed through the methods presented in this paper. The vertical length of a double category is meant to serve as a measure of how intricate the the relation between vertical and horizontal composition of squares in a double category is. We study the vertical length of double categories constructed through our methods.

\

\noindent We present the contents of this paper. In \textbf{section 2} we establish the notational and pictorial conventions used throughout the paper. In \textbf{section 3} we provide a detailed account of the construction of double categories of the form $C^\Phi$ and we prove that equations $H^*C=(\mathcal{B}^*,\mathcal{B})$ as above always admit solutions. In \textbf{section 4} we present examples of double categories constructed through the methods presented in section 3. We provide in particular the construction of a linear double category $C$ such that $H^*C$ is the pair formed by the category of von Neumann algebras and the bicategory of Hilbert bimodules. In \textbf{section 5} we study relations between double categories constructed through our methods and the conditions of having foldings and cofoldings. In \textbf{section 6} we prove that double categories constructed in section 3 all have vertical length 1. In \textbf{section 7} we study conditions for double categories constructed in section 2 that guarantee that these double categories are GG. Finally, in \textbf{section 8} we study the problem of universality of the construction presented in section 3 in the specific case of bicategories decorated by groups.

\section{Preliminaries}\label{prelim}

\noindent In this first section we establish the notational and pictorial conventions used throughout the paper. We briefly recall the Grothendieck construction, and we present a brief introduction to GG double categories and vertical length.

\newpage

\noindent\textit{bicategories}

\

\noindent Given a bicategory $\mathcal{B}$ we will write $\mathcal{B}_0,\mathcal{B}_1,\mathcal{B}_2$ for the collections of 0-,1-, and 2-cells of $\mathcal{B}$. We will write $dom^\mathcal{B},codom^\mathcal{B}$ for both the 1- and 2-dimensional domain and codomain functions of $\mathcal{B}$, we will write $id^\mathcal{B}$ for both the 0- and 1-dimensional identity functions of $\mathcal{B}$ and we will write $\circ,\circledast$ for the vertical and the horizontal composition operations in $\mathcal{B}$. We call $dom^\mathcal{B},codom^\mathcal{B},id^\mathcal{B},\circ,\circledast$ the structure data of $\mathcal{B}$. We call the identity transformations and associator the coherence data for $\mathcal{B}$. 0- and 1-cells in a double category are pictorially represented by points and horizontal arrows respectively and 2-morphisms are represented by globe diagrams read from top to bottom, i.e. a diagram as:

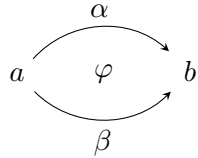
\begin{center}
\begin{tikzpicture}
  \matrix (m) [matrix of math nodes,row sep=4em,column sep=4em,minimum width=2em]
  {
     a&b \\};
  \path[-stealth]
    (m-1-1) edge [bend right=45] node [below] {$\beta$}(m-1-2)
            edge [bend left=45] node [above] {$\alpha$} (m-1-2)
            edge [white] node [black][fill=white] {$\varphi$} (m-1-2);
\end{tikzpicture}
\end{center}

\noindent represents a 2-cell from $\alpha$ to $\beta$. The vertical and horizontal composition operations in $\mathcal{B}$ are pictorially represented by vertical and horizontal concatenation of diagrams as above, see \cite{Borceux}. We will adopt the following non-standard convention for identitiy 1-and 2-cells: Given a 0-cell $a$ in a bicategory $\mathcal{B}$ we will represent the identity endomorphism of $a$ in $\mathcal{B}$ by a red arrow from $a$ to $a$. The following diagrams thus represent 2-cells from $id_a$ to $\beta$, from $\psi$ to $id_a$ and an endomorphism $\eta$ of $id_a$:

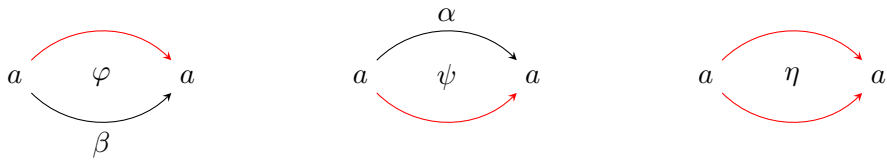
\begin{center}
\begin{tikzpicture}
  \matrix (m) [matrix of math nodes,row sep=4em,column sep=4em,minimum width=2em]
  {
     a&a&a&a&a&a \\};
  \path[-stealth]
    (m-1-1) edge [bend right=45] node [below] {$\beta$}(m-1-2)
            edge [red,bend left=45] node [above] {} (m-1-2)
            edge [white] node [black][fill=white] {$\varphi$} (m-1-2)
            
            (m-1-3) edge [red,bend right=45] node [below] {}(m-1-4)
            edge [bend left=45] node [above] {$\alpha$} (m-1-4)
            edge [white] node [black][fill=white] {$\psi$} (m-1-4)

            (m-1-5) edge [red,bend right=45] node [below] {}(m-1-6)
            edge [red,bend left=45] node [above] {} (m-1-6)
            edge [white] node [black][fill=white] {$\eta$} (m-1-6);
\end{tikzpicture}
\end{center}

\noindent Other conventions for the pictorial representation of identity endomorphisms in bicategories are \cite{DouglasHenriquesInternal}. We write \textbf{bCat} for the category of bicategories and pseudofunctors. Given a 0-cell $a$ in $\mathcal{B}$ the horizontal composition operation $\circledast$ in $\mathcal{B}$ provides the category of endomorphisms $End_\mathcal{B}(a)$ of $a$ in $\mathcal{B}$ with the structure of a monoidal category. The monoidal unit in $End_\mathcal{B}(a)$ is the identity $id^\mathcal{B}_a$. Conversely, every monoidal category $D$ defines a bicategory, which we will denote as $2D$, with a single object $\ast$. The horizontal composition operation $\circledast$ in $2D$ is the tensor product operation in $D$ and the identity 1-cell $id^\mathcal{B}_\ast$ of the single 0-cell in $\mathcal{B}$ is the monoidal identity \textbf{1}$_D$.

We will further adopt the following convention: Given a bicategory $\mathcal{B}$ we will write $E\mathcal{B}_2$ for the category whose collection of objects is the collection of 1-cells $\alpha$ in $\mathcal{B}$ satisfying the equation $dom^\mathcal{B}(\alpha)=codom^\mathcal{B}(\alpha)$ and whose morphisms are 2-cells in $\mathcal{B}$. Further, we will write $\tilde{B}_2$ for the category whose collection of objects is the compliment in $\mathcal{B}_1$ of the collection of objects in $E\mathcal{B}_2$ and whose morphisms are 2-cells in $\mathcal{B}$. The composition operation in both $E\mathcal{B}_2$ and $\tilde{B}_2$ is the vertical composition operation of $\mathcal{B}$.

\

\noindent \textit{Double categories}

\

\noindent Given a double category $C$ we will write $C_0,C_1$ for the category of objects and the category of morphisms of $C$ respectively, we will write $s,t,i$ for the source, target and horizontal identity functors of $C$, and following \cite{SchulmanDerived} we will write $\boxminus$ for the horizontal composition functor of $C$ and we will write $\boxvert$ for the vertical composition operation of both vertical morphisms and 2-morphisms in $C$. We call $s,t,i,\boxminus$ the structure data for $C$ and we call the identity transformations and associator the coherence data for $C$. We write \textbf{dCat} for the category of double categories and double functors. We represent objects in $C$ pictorially as verteces, horizontal and vertical morphisms as horizontal and vertical arrows and we represent 2-morphisms as squares. We read squares from top to bottom and from left to right, i.e. a diagram as:

\begin{center}
\begin{tikzpicture}
  \matrix (m) [matrix of math nodes,row sep=3em,column sep=3em,minimum width=2em]
  {
     a&b \\
     c&d \\};
  \path[-stealth]
    (m-1-1) edge node [above]{$\alpha$}(m-1-2)
            edge [white] node [black][fill=white]{$\varphi$}(m-2-2)    
    (m-2-1) edge node [below] {$\beta$} (m-2-2)
    (m-1-1) edge node [left]{$f$} (m-2-1)
    (m-1-2) edge node [right]{$g$} (m-2-2);
\end{tikzpicture}

\end{center}

\noindent represents a 2-morphism $\varphi$ such that $s(\varphi)=f,t(\varphi)=g$ and such that the vertical domain and codomain of $\varphi$ are $\alpha$ and $\beta$ respectively. We will adopt a similar non standard convention as the one considered in the previous section when pictorially representing identities. We will represent horizontal identities pictorially as horizontal red arrows. Thus diagrams of the form:

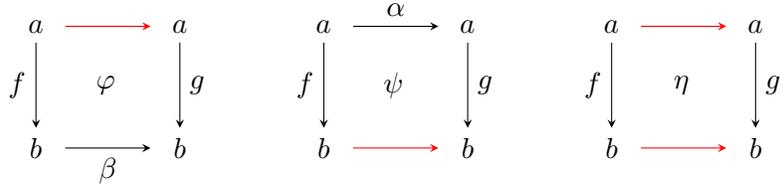
\begin{center}
\begin{tikzpicture}
  \matrix (m) [matrix of math nodes,row sep=3em,column sep=3em,minimum width=2em]
  {
     a&a&a&a&a&a \\
     b&b&b&b&b&b \\};
  \path[-stealth]
    (m-1-1) edge [red]node {}(m-1-2)
            edge [white] node [black][fill=white]{$\varphi$}(m-2-2)    
    (m-2-1) edge node [below] {$\beta$} (m-2-2)
    (m-1-1) edge node [left]{$f$} (m-2-1)
    (m-1-2) edge node [right]{$g$} (m-2-2)
    
   (m-1-3) edge node [above]{$\alpha$}(m-1-4)
            edge [white] node [black][fill=white]{$\psi$}(m-2-4)    
    (m-2-3) edge [red] node [below] {} (m-2-4)
    (m-1-3) edge node [left]{$f$} (m-2-3)
    (m-1-4) edge node [right]{$g$} (m-2-4)
    
    (m-1-5) edge [red] node {}(m-1-6)
            edge [white] node [black][fill=white]{$\eta$}(m-2-6)    
    (m-2-5) edge [red] node {} (m-2-6)
    (m-1-5) edge node [left]{$f$} (m-2-5)
    (m-1-6) edge node [right]{$g$} (m-2-6);
\end{tikzpicture}

\end{center}

\noindent represent squares $\varphi,\psi,\eta$ with vertical domain $i_a$, vertical codomain $i_b$ and vertical domain and codomain $i_a$ and $i_b$ respectively. We represent the horizontal identity $i_f$ of a vertical morphism $f:a\to b$ pictorially as:

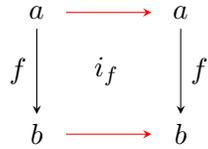
\begin{center}
\begin{tikzpicture}
  \matrix (m) [matrix of math nodes,row sep=3em,column sep=3em,minimum width=2em]
  {
     a&a \\
     b&b \\};
  \path[-stealth]
    (m-1-1) edge [red] node {}(m-1-2)
            edge [white] node [black][fill=white]{$i_f$}(m-2-2)    
    (m-2-1) edge [red] node {}  (m-2-2)
    (m-1-1) edge node [left]{$f$} (m-2-1)
    (m-1-2) edge node [right]{$f$} (m-2-2);
\end{tikzpicture}

\end{center}

\noindent Further, we represent vertical identities pictorially with blue vertical arrows. Thus diagrams of the form:

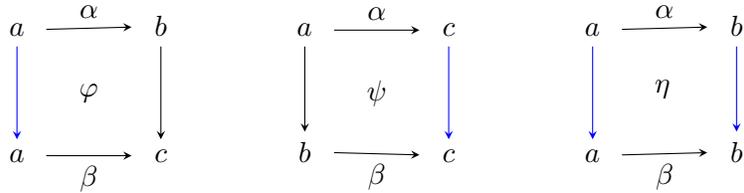
\begin{center}
\begin{tikzpicture}
  \matrix (m) [matrix of math nodes,row sep=3em,column sep=3em,minimum width=2em]
  {
     a&b&a&c&a&b \\
     a&c&b&c&a&b \\};
  \path[-stealth]
    (m-1-1) edge node [above]{$\alpha$}(m-1-2)
            edge [white] node [black][fill=white]{$\varphi$}(m-2-2)    
    (m-2-1) edge node [below] {$\beta$} (m-2-2)
    (m-1-1) edge [blue] node {} (m-2-1)
    (m-1-2) edge node {} (m-2-2)
    
   (m-1-3) edge node [above]{$\alpha$}(m-1-4)
            edge [white] node [black][fill=white]{$\psi$}(m-2-4)    
    (m-2-3) edge node [below] {$\beta$} (m-2-4)
    (m-1-3) edge node {} (m-2-3)
    (m-1-4) edge [blue] node {} (m-2-4)
    
    (m-1-5) edge node [above]{$\alpha$}(m-1-6)
            edge [white] node [black][fill=white]{$\eta$}(m-2-6)    
    (m-2-5) edge node [below]{$\beta$} (m-2-6)
    (m-1-5) edge [blue] node {} (m-2-5)
    (m-1-6) edge [blue] node {} (m-2-6);
\end{tikzpicture}

\end{center}

\noindent represent squares $\varphi,\psi,\eta$ such that $s(\varphi)=id_a,t(\psi)=id_c$ and $s(\eta)=id_a,t(\eta)=id_b$. We represent the vertical and horizontal composition operations $\boxvert,\boxminus$ pictorially as vertical and horizontal concatenation respectively. With our conventions globular squares are precisely those 2-morphisms admitting pictorial representations as:

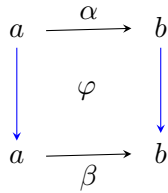
\begin{center}
\begin{tikzpicture}
  \matrix (m) [matrix of math nodes,row sep=3em,column sep=3em,minimum width=2em]
  {
     a&b \\
     a&b \\};
  \path[-stealth]
    (m-1-1) edge node [above]{$\alpha$}(m-1-2)
            edge [white] node [black][fill=white]{$\varphi$}(m-2-2)    
    (m-2-1) edge node [below] {$\beta$} (m-2-2)
    (m-1-1) edge [blue] node [left]{} (m-2-1)
    (m-1-2) edge [blue] node [right]{} (m-2-2);
\end{tikzpicture}

\end{center}

\

\noindent \textit{Decorated bicategories}

\

\noindent Given a double category $C$ we write $HC$ for the bicategory whose 0-, 1-, and 2-cells are objects, horizontal morphisms, and globular squares of $C$ respectively. The operation $C\mapsto HC$ 'flattens' the double category $C$ into a bicategory by interpreting diagrams in $C$ of the form:

\begin{center}
\begin{tikzpicture}
  \matrix (m) [matrix of math nodes,row sep=3em,column sep=3em,minimum width=2em]
  {
     a&b \\
     a&b \\};
  \path[-stealth]
    (m-1-1) edge node [above]{$\alpha$}(m-1-2)
            edge [white] node [black][fill=white]{$\varphi$}(m-2-2)    
    (m-2-1) edge node [below] {$\beta$} (m-2-2)
    (m-1-1) edge [blue] node [left]{} (m-2-1)
    (m-1-2) edge [blue] node [right]{} (m-2-2);
\end{tikzpicture}

\end{center}

\noindent as diagrams of the form:

\begin{center}
\begin{tikzpicture}
  \matrix (m) [matrix of math nodes,row sep=4em,column sep=4em,minimum width=2em]
  {
     a&b \\};
  \path[-stealth]
    (m-1-1) edge [bend right=45] node [below] {$\beta$}(m-1-2)
            edge [bend left=45] node [above] {$\alpha$} (m-1-2)
            edge [white] node [black][fill=white] {$\varphi$} (m-1-2);
\end{tikzpicture}
\end{center}

\noindent Given a bicategory $\mathcal{B}$ we will say that a category $\mathcal{B}^*$ is a decoration of $\mathcal{B}$ if the collection of objects of $\mathcal{B}^*$ is the collection of 0-cells of $\mathcal{B}$. In this case we will say that the pair $(\mathcal{B}^*,\mathcal{B})$ is a decorated bicategory. The pair $(C_0,HC)$ is a decorated bicategory for every double category $C$. We write $H^*C$ for this decorated bicategory. We call $H^*C$ the decorated horizontalization of $C$. We write \textbf{bCat}$^*$ for the category of decorated bicategories and decorated pseudofunctors, where given decorated bicategories $(\mathcal{B}^*,\mathcal{B}),(\mathcal{B}'^*,\mathcal{B}')$ we understand for a decorated pseudofunctor from $(\mathcal{B}^*,\mathcal{B})$ to $(\mathcal{B}'^*,\mathcal{B}')$ a pair $(F^*,F)$ where $F^*:\mathcal{B}^*\to\mathcal{B}'^*$ and $F:\mathcal{B}\to \mathcal{B}'$ and such that $F^*,F$ coincide in $\mathcal{B}_0$. We are interested in the following problem:

\begin{prob}\label{prob}
Let $(\mathcal{B}^*,\mathcal{B})$ be a decorated bicategory. Find double categories $C$ satisfying the equation $H^*C=(\mathcal{B}^*,\mathcal{B})$.
\end{prob}

\noindent We understand problem \ref{prob} as the problem of lifting a bicategory $\mathcal{B}$ to a double category through an orthogonal direction, provided by $\mathcal{B}^*$. We call solutions to this problem internalizations of $(\mathcal{B}^*,\mathcal{B})$. Pictorially, solutions to problem \ref{prob} are to be understood as ways to formally understand diagrams of the form:

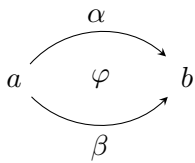
\begin{center}
\begin{tikzpicture}
  \matrix (m) [matrix of math nodes,row sep=4em,column sep=4em,minimum width=2em]
  {
     a&b \\};
  \path[-stealth]
    (m-1-1) edge [bend right=45] node [below] {$\beta$}(m-1-2)
            edge [bend left=45] node [above] {$\alpha$} (m-1-2)
            edge [white] node [black][fill=white] {$\varphi$} (m-1-2);
\end{tikzpicture}
\end{center}

\noindent as diagrams of the form:

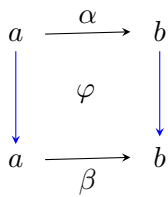
\begin{center}
\begin{tikzpicture}
  \matrix (m) [matrix of math nodes,row sep=3em,column sep=3em,minimum width=2em]
  {
     a&b \\
     a&b \\};
  \path[-stealth]
    (m-1-1) edge node [above]{$\alpha$}(m-1-2)
            edge [white] node [black][fill=white]{$\varphi$}(m-2-2)    
    (m-2-1) edge node [below] {$\beta$} (m-2-2)
    (m-1-1) edge [blue] node [left]{} (m-2-1)
    (m-1-2) edge [blue] node [right]{} (m-2-2);
\end{tikzpicture}

\end{center}

\noindent within a double category, taking non-identity, i.e. non-blue vertical morphims into account. See \cite{Bartels1,Bartels2} for applications of problem \ref{prob} in the case of decorated bicategories of von Neumann algebras to the theory of invariants of manifolds of dimension 3.


\

\noindent \textit{The Grothendieck construction}

\

\noindent We use the Grothendieck construction to provide solutions to problem \ref{prob}. The initial input of the Grothendieck construction is a category $C$ and a functor $\Phi:C\to \mbox{\textbf{Cat}}$. The Grothendieck construction associates to $\Phi$ a new category $\int_C{\Phi}$ which we now describe.

\begin{definition}\label{Grothendieckdefi}
Let $C$ be a category. Let $\Phi:C\to \mbox{\textbf{Cat}}$ be a functor. The category $\int_C{\Phi}$ is defined as follows:

\begin{enumerate}
    \item \textbf{Objects:} The collection of objects of $\int_C\Phi$ is the collection of all pairs $(x,a)$ where $x$ is a object in $C$ and $a$ is an object in $\Phi(x)$. 
    \item \textbf{Morphisms:} Let $(x,a),(x',a')$ be objects in $\int_C\Phi$. The collection of morphisms, in $\int_C\Phi$, from $(x,a)$ to $(x',a')$ is the collection of all pairs $(\alpha,\beta)$ where $\alpha$ is a morphism, in $C$, from $x$ to $x'$ and where $\beta$ is a morphism, in $\Phi(x')$, from $\Phi(\alpha)(x)$ to $x'$.
    
\item \textbf{Composition:} Let $(x,a),(x',a'),(x'',a'')$ be objects in $\int_C{\Phi}$. Let $(\alpha,\beta)$ and $(\alpha',\beta')$ be morphisms, in $\int_C\Phi$, from $(x,a)$ to $ (x'.a')$ and from $(x',a')\to (x'',a'')$ respectively. The composition  $(\alpha',\beta')(\alpha,\beta)$ is defined by the equation

\[(\alpha',\beta')(\alpha,\beta)=(\alpha'\alpha,\beta'\Phi(\alpha')(\beta))\]
\end{enumerate} 
\end{definition}

\

\noindent \textit{The GG piece and vertical length}

\

\noindent We briefly recall the basics of the theory of globularily generated double categories and their vertical length. We refer the reader to \cite{yo1,yo2} for an account on the subject. We say that a double category $C$ is globularily generated or GG for short, if it is generated, as a double category, by its collection of globular squares, i.e. a double category $C$ is GG if $C$ is generated by squares of the form:

\begin{center}
\begin{tikzpicture}
  \matrix (m) [matrix of math nodes,row sep=3em,column sep=3em,minimum width=2em]
  {
     a&b&a&a \\
     a&b&b&b \\};
  \path[-stealth]
    (m-1-1) edge [red]node {}(m-1-2)
            edge [white] node [black][fill=white]{$\varphi$}(m-2-2)    
    (m-2-1) edge [red] node {} (m-2-2)
    (m-1-1) edge node [left]{$f$} (m-2-1)
    (m-1-2) edge node [right]{$g$} (m-2-2)
    
   (m-1-3) edge node [above]{$\alpha$}(m-1-4)
            edge [white] node [black][fill=white]{$\psi$}(m-2-4)    
    (m-2-3) edge node [below] {$\beta$} (m-2-4)
    (m-1-3) edge [blue] node {} (m-2-3)
    (m-1-4) edge [blue] node {} (m-2-4);
\end{tikzpicture}

\end{center}

\noindent The category of morphisms of a GG category $C$ admits a presentation as a limit $C_1=\varinjlim V^k_C$ where the categories $V^k_C$ are defined inductively by setting $V^0_C=HC_1$ and by making $V^k_C$ be the category generated by horizontal compositions of morphisms in $V^{k-1}_C$ for every $k\geq 1$. We define the vertical length $\ell C$ of $C$ as the minimal $k\geq 1$ such that $C_1=V^k_C$ if such $k$ exists. We make $\ell C=\infty$ otherwise. The vertical length of a GG categories is to be understood as a measure of complexity on the relations between the vertical and horizontal composition operations $\boxvert,\boxminus$ in $C$.

For a general i.e. non-GG double category we write $\gamma C$ for the sub-double category of $C$ generated by 2.morphisms admitting a diagramatic representation of the form:

\begin{center}
\begin{tikzpicture}
  \matrix (m) [matrix of math nodes,row sep=3em,column sep=3em,minimum width=2em]
  {
     a&b&a&a \\
     a&b&b&b \\};
  \path[-stealth]
    (m-1-1) edge [red]node {}(m-1-2)
            edge [white] node [black][fill=white]{$\varphi$}(m-2-2)    
    (m-2-1) edge [red] node {} (m-2-2)
    (m-1-1) edge node [left]{$f$} (m-2-1)
    (m-1-2) edge node [right]{$g$} (m-2-2)
    
   (m-1-3) edge node [above]{$\alpha$}(m-1-4)
            edge [white] node [black][fill=white]{$\psi$}(m-2-4)    
    (m-2-3) edge node [below] {$\beta$} (m-2-4)
    (m-1-3) edge [blue] node {} (m-2-3)
    (m-1-4) edge [blue] node {} (m-2-4);
\end{tikzpicture}

\end{center}

\noindent Thus defined $\gamma C$ is such that $H^*\gamma C=H^*C$. Moreover $\gamma C$ is contained in every sub-double category of $C$ satisfying this condition. Double categories of the form $\gamma C$ are always GG and every GG double category is of this form. GG categories are thus considered as minimal solutions to problem \ref{prob}. We regard the theory of GG categories as lying 'in between' the theory of double cateogries and the theory of bicategories. We define the vertical length $\ell C$ of a non-GG double category $C$ by $\ell \gamma C$.

\

\section{The main construction}

\noindent We use a special instance of the Grothendieck construction in order to produce internalizations of decorated bicategories. The special situation we consider is as follows:

\

\noindent Let $(\mathcal{B}^*,\mathcal{B})$ be a decorated bicategory. We consider the Grothendieck construction with input data given by a functor $\Phi:\mathcal{B}^*\to \mbox{\textbf{Cat}}$ such that for every object $a$ of $\mathcal{B}^*$ the category $\Phi(a)$ is equal to $End_\mathcal{B}(a)$ and such that for every pair of objects $a,b$ in $\mathcal{B}^*$ and for every $\alpha:a\to b$ in $\mathcal{B}^*$ the functor $\Phi_\alpha:End_\mathcal{B}(a)\to End_\mathcal{B}(b)$ is monoidal. That is, we consider the Grothendieck construction $\int_{\mathcal{B}^*}\Phi$ for functorial extensions $\Phi:\mathcal{B}^*\to\mbox{\textbf{Cat}}^\otimes$ of the function associating to every object $a$ in $\mathcal{B}^*$ the endomorphism category $End_\mathcal{B}(a)$. For our purposes we consider the following extension of $\int_{\mathcal{B}^*}\Phi$ in the situation described above: We write $\int^\ast_{\mathcal{B}^*}\Phi$ for the category obtained as the disjoint union of $\int_{\mathcal{B}^*}\Phi$ and $\tilde{B}_2$. We prove the following theorem.


\

\begin{thm}\label{thmmain}
In the situation above the pair $(\mathcal{B}^*,\int^*_{\mathcal{B}^*}\Phi)$ admits the structure of a double category. Writing $C^\Phi$ for this double category, $C^\Phi$ satisfies the equation:

\[H^*C^\Phi=(\mathcal{B}^*,\mathcal{B})\]

\end{thm}

\begin{proof}
Let $(\mathcal{B}^*,\mathcal{B})$ be a decorated bicategory. Let $\Phi:\mathcal{B}^*\to\mbox{\textbf{Cat}}^\otimes$ be a functor such that $\Phi(a)$ is equal to $End_\mathcal{B}(a)$ for every object $a$ in $\mathcal{B}^*$. We wish to prove that in this case the pair $C^\Phi=(\mathcal{B}^*,\int^*_{\mathcal{B}^*}\Phi)$ admits a structure of double category such that with this structure $C^\Phi$ internalizes $(\mathcal{B}^*,\mathcal{B})$. We begin by defining the structure functors of $C^\Phi$.

First observe that the collection of objects of $\int_{\mathcal{B}^*}\Phi$ is the collection of all pairs of the form $(a,\alpha)$ where $a$ is a 0-cell of $\mathcal{B}$ and where $\alpha$ is an endomorphism of $a$ in $\mathcal{B}$. Identifying every pair $(a,\alpha)$ with the endomorphism $\alpha$ we identify the collection of objects of $\int_{\mathcal{B}^*}\Phi$ with the collection of 1-cells $\alpha$ of $\mathcal{B}$ such that $s\alpha=t\alpha$. That is, we identify the collection of objects of $\int_{\mathcal{B}^*}\Phi$ with the compliment in $\mathcal{B}_1$ of the collection of objects of $\tilde{B}_1$. The collection of objects of $\int^*_{\mathcal{B}^*}\Phi$ can thus be identified with $\mathcal{B}_1$. We assume this identification has been performed and we thus say that the collection of objects of $\int^*_{\mathcal{B}^*}\Phi$ is equal to $\mathcal{B}_1$.

We define source and target functors $s,t:\int^*_{\mathcal{B}^*}\Phi\to \mathcal{B}^*$ for $C^\Phi$. We make the object functions of $s,t$ to be equal to the restrictions, to $\mathcal{B}_1$, of the domain and codomain functions $dom^\mathcal{B},codom^\mathcal{B}$ functions of $\mathcal{B}$. Let $(f,\varphi)$ be a morphism in $\int_{\mathcal{B}^*}\Phi$. In that case we make both $s(f,\varphi)$ and $t(f,\varphi)$ to be equal to the morphism $f$ in $\mathcal{B}^*$. Let now $\varphi$ be a morphism in $\tilde{\mathcal{B}}_1$. Suppose that $\varphi$ admits a pictorial representation, in $\mathcal{B}$, as:

\begin{center}
\begin{tikzpicture}
  \matrix (m) [matrix of math nodes,row sep=4em,column sep=4em,minimum width=2em]
  {
     a&b \\};
  \path[-stealth]
    (m-1-1) edge [bend right=45] node [below] {$\beta$}(m-1-2)
            edge [bend left=45] node [above] {$\alpha$} (m-1-2)
            edge [white] node [black][fill=white] {$\varphi$} (m-1-2);
\end{tikzpicture}
\end{center}

\noindent In that case we make $s(\varphi)=id_a$ and $t(\varphi)=id_b$. The fact that thus defined $s,t$ indeed define functors from $\int^\ast_{\mathcal{B}^*}\Phi$ to $\mathcal{B}^*$ follows from the way composition in $\int_{\mathcal{B}^*}\Phi$ is defined, see section \ref{prelim}, and by the fact that $\mathcal{B}$ is a bicategory. We thus represent morphisms $\varphi$ in $\tilde{B}_1$ and morphisms $(f,\varphi)$ in $\int_{\mathcal{B}^*}\Phi$ pictorially as:

\begin{center}
\begin{tikzpicture}
  \matrix (m) [matrix of math nodes,row sep=3em,column sep=3em,minimum width=2em]
  {
     a&b & b&b \\
     a&b & a&a \\};
  \path[-stealth]
    (m-1-1) edge node [above]{$\beta$}(m-1-2)
            edge [white] node [black][fill=white]{$\varphi$}(m-2-2)    
    (m-2-1) edge node [below] {$\alpha$} (m-2-2)
    (m-1-1) edge [blue] node {} (m-2-1)
    (m-1-2) edge [blue] node {} (m-2-2)
    
    (m-1-3) edge node [above]{$\beta$}(m-1-4)
            edge [white] node [black][fill=white]{$(f,\varphi)$}(m-2-4)    
    (m-2-3) edge node [below] {$\alpha$} (m-2-4)
    (m-1-3) edge node [left]{$f$} (m-2-3)
    (m-1-4) edge node [right]{$f$} (m-2-4);
\end{tikzpicture}

\end{center}

\noindent respectively. We now define a horizontal identity funtor for $C^\Phi$.

Let $a$ be an object of $\mathcal{B}^*$. We make the horizontal identity $i_a$ of $a$, in $C^\Phi$, to be equal to the horizontal identity $id^\mathcal{B}_a$ of $a$ in $\mathcal{B}$. Let $f:a\to b$ be a morphism in $\mathcal{B}^*$. In that case we make $i_f$ to be equal to the morphism $(f,id^\mathcal{B}_{i_b})$ of $\int^*_{\mathcal{B}^*}\Phi$. The fact that thus defined $i$ indeed defines a functor from $\mathcal{B}^\ast$ to $\int^\ast_{\mathcal{B}^\ast}\Phi$ again easily follows from the way the composition operation in $\int_{\mathcal{B}^*}\Phi$ is defined. An easy check proves that $i$ is compatible with the functors $s,t$ defined above. We represent the horizontal identity $(f,id^\mathcal{B}_{i_b})$ of a morphism $f:a\to b$ in $\mathcal{B}^*$ pictorially as:

\begin{center}
\begin{tikzpicture}
  \matrix (m) [matrix of math nodes,row sep=3em,column sep=3em,minimum width=2em]
  {
     b&b \\
     a&a \\};
  \path[-stealth]
    (m-1-1) edge [red] node {}(m-1-2)
            edge [white] node [black][fill=white]{$i_f$}(m-2-2)    
    (m-2-1) edge [red] node {} (m-2-2)
    (m-1-1) edge node [left]{$f$} (m-2-1)
    (m-1-2) edge node [right]{$f$} (m-2-2);
\end{tikzpicture}

\end{center}

\noindent In order to define a horizontal composition functor on $C^\Phi$ we first introduce a notational modification on the collection of morphisms of $\int^*_{\mathcal{B}^*}\Phi$. We will write $(f,f,\varphi)$ for every morphsim $(f,\Phi)$ in $\int_{\mathcal{B}^*}\Phi$ and we will write $(id_a,id_b,\varphi)$ for every 2-cell $\varphi$ in $\tilde{\mathcal{B}}_1$ such that the 0-dimensional domain and codomain of $\varphi$ are equal to $a$ and $b$ respectively. Under this convention the collection of morphisms of $\int^*_{\mathcal{B}^*}\Phi$ is the collection of all triples $(f,g,\varphi)$ where $f,g$ are morphisms in $\mathcal{B}^*$ and where $\varphi$ is a 2-cell in $\mathcal{B}$. If the morphisms $f,g$ are not the identities, in $\mathcal{B}^*$, of the 0-dimensional domain and codomain of $\varphi$ respectively, then $f$ and $g$ are equal.

With this notational convention in place we now define a horizontal composition functor $\boxminus$ for $C^\Phi$. We make $\boxminus$ to be defined as the horizontal composition operation of $\mathcal{B}$ on the collection of objects of $\int^*_{\mathcal{B}^*}\Phi$. In order to define an extension of this operation to a functor $\boxminus:\int^*_{\mathcal{B}^*}\Phi\times_{\mathcal{B}^*}\int^*_{\mathcal{B}^*}\Phi\to\int^*_{\mathcal{B}^*}\Phi$ we first analize the morphism of $\int^*_{\mathcal{B}^*}\Phi\times_{\mathcal{B}^*}\int^*_{\mathcal{B}^*}\Phi$. Let $(f_i,g_i,\varphi_i)$ with $i=1,2$ be a morphism in $\int^*_{\mathcal{B}^*}\Phi\times_{\mathcal{B}^*}\int^*_{\mathcal{B}^*}\Phi$. Observe that by the way the functors $s,t$ were defined if one of $f_1,f_2,g_1,g_2$ is not an identity morphism of $\mathcal{B}^*$ then none of $f_1,f_2,g_1,g_2$ are identity morphisms in $\mathcal{B}^*$ and moreover in that case $f_1,f_2,g_1,g_2$ are all equal. A pair $(f_i,g_i,\varphi_i)$ of morphisms in $\int^*_{\mathcal{B}^*}\Phi$ is thus a morphism in $\int^*_{\mathcal{B}^*}\Phi\times_{\mathcal{B}^*}\int^*_{\mathcal{B}^*}\Phi$ if either all of the $f_i,g_i$ with $i=1,2$ are equal or all of the $f_i,g_i$ with $i=1,2$ are identity morphisms in $\mathcal{B^*}$, in which case the pair $(f_i,g_i,\varphi_i)$ with $i=1,2$ is actually horizontally composable in $\mathcal{B}$, i.e. is a 2-cell in $\mathcal{B}_2\times_{\mathcal{B}_0}\mathcal{B}_2$. We define the horizontal composition of pairs satisfying this last condition as their horizontal composition in $\mathcal{B}$. In order to define a horizontal composition functor $\boxminus$ we now need only to define the horizontal composition of morhisms in $\int_{\mathcal{B}^*}\Phi\times_{\mathcal{B}^*}\int_{\mathcal{B}^*}¨\Phi$. Let $(f,\varphi),(f,\psi)$ be a pair in $\int_{\mathcal{B}^*}\Phi\times_{\mathcal{B}^*}\int_{\mathcal{B}^*}¨\Phi$. We make the horizontal composition $(f,\varphi)\boxminus(f,\psi)$ to be the pair $(f,\varphi\circledast\psi)$ where $\phi\circledast\psi$ denotes the horizontal composition, in $\mathcal{B}$, of the composable par of 2-cells $(\varphi,\psi)$. We prove that thus defined the operation of horizontal composition $\boxminus$ does indeed define a functor $\int^*_{\mathcal{B}^*}\Phi\times_{\mathcal{B}^*}\int^*_{\mathcal{B}^*}\Phi\to \int^*_{\mathcal{B}^*}\Phi$.

In order to prove the functoriality of the operation $\boxminus$ defined above we need to prove that the two possible compositions of compatible squares as in the following diagram:

\begin{center}
\begin{tikzpicture}
  \matrix (m) [matrix of math nodes,row sep=3em,column sep=3em,minimum width=2em]
  {
     a&b & b&c \\
     a'&b' & b'&c' \\
      a'&b'& b'&c' \\
     a''&b''& b''&c''\\};
  \path[-stealth]
    (m-1-1) edge node {}(m-1-2)
            edge [white]node [black][fill=white]{$\varphi$}(m-2-2)    
    (m-2-1) edge node [below] {$\alpha$} (m-2-2)
    (m-1-1) edge node {} (m-2-1)
    (m-1-2) edge node [right] {$f$} (m-2-2)
    
    (m-1-3) edge node [above]{}(m-1-4)
            edge [white]node [black][fill=white]{$\psi$}(m-2-4)    
    (m-2-3) edge node [below] {$\beta$} (m-2-4)
    (m-1-3) edge node [left]{$f$} (m-2-3)
    (m-1-4) edge node {} (m-2-4)

    (m-3-1) edge node [above]{$\alpha$}(m-3-2)
            edge [white]node [black][fill=white]{$\varphi'$}(m-4-2)    
    (m-4-1) edge node {} (m-4-2)
    (m-3-1) edge node {} (m-4-1)
    (m-3-2) edge node [right] {$g$} (m-4-2)
    
    (m-3-3) edge node [above]{$\beta$}(m-3-4)
            edge [white]node [black][fill=white]{$\psi'$}(m-4-4)    
    (m-4-3) edge node {} (m-4-4)
    (m-3-3) edge node [left]{$g$} (m-4-3)
    (m-3-4) edge node {} (m-4-4);
\end{tikzpicture}

\end{center}

\noindent are equal. From the fact that $\mathcal{B}$ is a bicategory and thus satisfies the exchange lemma the above diagrammatic equation is true whenever the morphisms involved are 2-cells in $\mathcal{B}$, i.e. when all the $f_i,g_i$ with $i=1,2,3$ are identity morphisms in $\mathcal{B}^*$. We now prove that the above equation is true for morphisms in $\int_{\mathcal{B}^*}\Phi$. Let $(f,\varphi),(f,\psi)$ and $(f',\varphi'),(f',\psi')$ be morphisms in $\int_{\mathcal{B}^*}\Phi$ fitting in the diagram:

\begin{center}
\begin{tikzpicture}
  \matrix (m) [matrix of math nodes,row sep=3em,column sep=3em,minimum width=2em]
  {
     a&a & a&a \\
     b&b & b&b \\
     b&b & b&b \\
     c&c& c&c\\};
  \path[-stealth]
    (m-1-1) edge node [above]{$\alpha$}(m-1-2)
            edge [white]node [black][fill=white]{$(f,\varphi)$}(m-2-2)    
    (m-2-1) edge node [below] {$\beta$} (m-2-2)
    (m-1-1) edge node [left]{$f$} (m-2-1)
    (m-1-2) edge node [right] {$f$} (m-2-2)
    
    (m-1-3) edge node [above]{$\gamma$}(m-1-4)
            edge [white]node [black][fill=white]{$(f,\psi)$}(m-2-4)    
    (m-2-3) edge node [below] {$\eta$} (m-2-4)
    (m-1-3) edge node [left]{$f$} (m-2-3)
    (m-1-4) edge node [right]{$f$} (m-2-4)

    (m-3-1) edge node [above]{$\beta$}(m-3-2)
            edge [white]node [black][fill=white]{$(f',\varphi')$}(m-4-2)    
    (m-4-1) edge node [below]{$\nu$} (m-4-2)
    (m-3-1) edge node [left]{$f'$} (m-4-1)
    (m-3-2) edge node [right] {$f'$} (m-4-2)
    
    (m-3-3) edge node [above]{$\eta$}(m-3-4)
            edge [white]node [black][fill=white]{$(f'\psi')$}(m-4-4)    
    (m-4-3) edge node [below]{$\epsilon$} (m-4-4)
    (m-3-3) edge node [left]{$f'$} (m-4-3)
    (m-3-4) edge node [right]{$f'$} (m-4-4);
\end{tikzpicture}

\end{center}

\noindent We wish to prove that in this case the two possible compositions of the diagramatic scheme above are equal. That is, we wish to prove that the following equation holds:

\[ [(f',\varphi')\boxminus(f',\psi')]\boxvert[(f,\varphi)\boxminus(f,\psi)]=
 [(f',\varphi')\boxvert(f,\varphi)]\boxminus[(f',\psi')\boxvert(f,\psi)]\]

\noindent The left hand side of the above equation is equal to the composition 

\[(f',\varphi'\circledast\psi')\boxvert(f,\varphi\circledast\psi)\]

\noindent which in turn is equal to the pair

\[(f'f,(\varphi'\circledast\psi')\Phi_{f'}(\varphi\circledast \psi))\]

\noindent Now the right hand side of the above equation is equal to the horizontal composition

\[(f'f,\varphi'\Phi_{f'}\varphi)\boxminus(f'f,\psi'\Phi_{f'}\psi)\]

\noindent which in turn is equal to the pair

\[(f'f,[\varphi'\Phi_{f'}\varphi]\circledast[\psi'\Phi_{f'}\psi])\]

\noindent In order to prove the desired equation we thus need to prove the equality:

\[(\varphi'\circledast \psi')\Phi_{f'}(\varphi\circledast\psi)=[\varphi'\Phi_{f'}\varphi]\circledast[\psi'\Phi_{f'}\psi]\]

\noindent but this follows from the fact that $\Phi$ is a functor on \textbf{Cat}$^\otimes$. The horizontal composition operation $\boxminus$ defined above is thus functorial not only on $\tilde{\mathcal{B}}_2$ but is also functorial on $\int_{\mathcal{B}^*}\Phi$. We conclude that the horizontal composition operation $\boxminus$ defined above is a functor from $\int^*_{\mathcal{B}^*}\Phi\times_{\mathcal{B}^*}\int^*_{\mathcal{B}^*}\Phi$ to $\int^*_{\mathcal{B}^*}\Phi$.

From the way they were defined it is easily seen that the surce and target functors $s$ and $t$ of $C^\Phi$ are compatible with the horizontal composition functor $\boxminus$ of $C^\Phi$. Moreover, the left and right unit transformations and the associator of $C$ are easily seen to lift to compatible left and right horizontal identity transformations and associators for $C^\Phi$. We conclude that as defined above $C^\Phi$ is a double category. 

Finally, the category of objects $C^\Phi_0$ of $C^\Phi$ is equal to $\mathcal{B}^*$, the collection of objects of the category of morphisms $C^\Phi_1$ of $C^\Phi$ is the collection of 1-cells $\mathcal{B}_1$ of $\mathcal{B}$. Now the collection of globular squares of $C^\Phi$ is equal to the set of squares $(f,\varphi)$ in $\int^*_{\mathcal{B}^*}\Phi$ such that both $s(f,\varphi)$ and $t(f,\varphi)$ are identitiy morphisms in $\mathcal{B}^*$. This set is precisely the collection of 2-cells $\mathcal{B}_2$ of $\mathcal{B}$. By the way the double category $C^\Phi$ was constructed the restriction of the structure functors defining $C^\Phi$, i.e. source, target, horizontal identity and horizontal composition functors, to cells in $\mathcal{B}$ are precisely the corresponding structure functions for $\mathcal{B}$.  This proves that $C^\Phi$ satisfies the equation $H^*C^\Phi=(\mathcal{B}^*,\mathcal{B})$. This concludes the proof.

\end{proof}

\noindent The following corollary to theorem \ref{thmmain} says that problem \ref{prob} admits solutions for every decorated bicategory $(\mathcal{B}^*,\mathcal{B})$.

\begin{cor}
Let $(\mathcal{B}^*,\mathcal{B})$ be a decorated bicategory. There exists a double category $C$ satisfying the equation $H^*C=(\mathcal{B}^*,\mathcal{B})$.
\end{cor}

\begin{proof}
Let $(\mathcal{B}^*,\mathcal{B})$ be a decorated bicategory. We wish to prove that there exists a double cateogry $C$ satisfying the equation $H^*C=(\mathcal{B}^*,\mathcal{B})$. 

We prove that $Fun(\mathcal{B}^*,\mbox{\textbf{Cat}}^\otimes)_\mathcal{B}$ is non-empty. We do this by exhibiting a $\Phi$ in $Fun(\mathcal{B}^*,\mbox{\textbf{Cat}}^\otimes)_\mathcal{B}$. We make $\Phi$ be such that $\Phi(a)=End_\mathcal{B}(a)$ for every object $a$ in $\mathcal{B^*}$. Let $a,b$ be objects in $\mathcal{B}^*$. Let $f:a\to B$ be a morphism in $\mathcal{B}^*$. We make $\Phi_f:End_\mathcal{B}(a)\to End_\mathcal{B}(b)$ be the constant functor on the identity endomorphism $id^\mathcal{B}_b$. Thus defined $\Phi_f$ is clearly monoidal. Moreover, the assignment $f\mapsto\Phi_f$ is clearly compatible with the category structure of $\mathcal{B}^*$. The double category $C^\Phi$ associated to $\Phi$ by theorem \ref{thmmain} satisfies the equation $H^*C^\Phi=(\mathcal{B}^*,\mathcal{B})$. This concludes the proof. 
\end{proof}

\noindent Given a decorated bicategory $(\mathcal{B}^*,\mathcal{B})$ we will write $Fun(\mathcal{B}^*,\mbox{\textbf{Cat}}^\otimes)_{\mathcal{B}}$ for the category whose collection of objects is the collection of all functors satisfying the conditions in the statement of theorem \ref{thmmain}, i.e. the collection of objects of $Fun(\mathcal{B}^*,\mbox{\textbf{Cat}}^\otimes)_{\mathcal{B}}$ is the collection of all functors $\Phi:\mathcal{B}^*\to \mbox{\textbf{Cat}}^\otimes$ such that $\Phi(a)=End_{\mathcal{B}}(a)$ for every 0-cell $a$ of $\mathcal{B}$, and whose collection of morphisms is the collection of natural transformations between such functors. Further, we will write \textbf{dCat}$_{(\mathcal{B}^*,\mathcal{B})}$ for the category whose objects are internalizations of $(\mathcal{B}^*,\mathcal{B})$, i.e. the objects of \textbf{dCat}$_{(\mathcal{B}^*,\mathcal{B})}$ are double categories $C$ satisfying the equation $H^*C=(\mathcal{B}^*,\mathcal{B})$, and whose morphisms are double functors $F$ such that $H^*F=id_{(\mathcal{B}^*,\mathcal{B})}$. The Grothendieck construction admits an obvious functorial extension (a 2-functorial extension in fact) to a category of Grothendieck fibrations. It is not diffucult to see that this functor extends the construction of the double categories $C^\Phi$ presented in theorem \ref{thmmain} to a functor $C^\bullet:Fun(\mathcal{B}^*,\mbox{\textbf{Cat}}^\otimes)_{\mathcal{B}}\to \mbox{\textbf{dCat}}_{(\mathcal{B}^*,\mathcal{B})}$. We make use of this only in section 6 in a specific case. We are interested in the construction of $C^\Phi$ when both $\mathcal{B}^*$ and $\mathcal{B}$ are linear. With this in mind we make the following observation:

\begin{obs}
Let $k$ be a field. We will say that a bicategory $\mathcal{B}$ is linear over $k$ if for every $\alpha,\beta\in\mathcal{B}_1$ with the same source and target, the collection of 2-cells Hom$_{\mathcal{B}}(\alpha,\beta)$ from $\alpha$ to $\beta$ in $\mathcal{B}$ is endowed with the structure of a vector space over $k$ in such a way that the vertical and the horizontal composition operations in $\mathcal{B}$ are $k$-bilinear and such that the coherence data of $\mathcal{B}$ is linear. We will say that a decorated bicategory $(\mathcal{B}^*,\mathcal{B})$ is linear if both $\mathcal{B}^*$ and $\mathcal{B}$ are linear. Observe that given a linear bicategory $\mathcal{B}$ the horizontal composition operation in $\mathcal{B}$ provides the category of endomorphisms $End_\mathcal{B}(a)$ of any 0-cell $a$ in $\mathcal{B}$ with the structure of a linear tensor category. We will say that a double category $C$ is linear if $C_0,C_1$ are both linear and if the structure data $s,t,i,\boxminus$ of $C$ are linear. We write \textbf{Cat}$^\otimes_k$ for the category whose objects are linear tensor categories and whose morphisms are linear strict tensor functors. Given a linear decorated bicategory $(\mathcal{B}^*,\mathcal{B})$ we write \textbf{dCat}$^k_{(\mathcal{B}^*,\mathcal{B})}$ for the category whose objects are linear double categories $C$ satisfying $H^*C=(\mathcal{B}^*,\mathcal{B})$ and whose morphisms are double functors $F$ with linear objects and morphism functors $F_0,F_1$ such that $H^*F=id_{(\mathcal{B}^*,\mathcal{B})}$. It is easily seen that the construction presented in theorem \ref{thmmain} extends to a functor $C^\bullet:Fun(\mathcal{B}^*,\mbox{\textbf{Cat}}^\otimes_k)\to \mbox{\textbf{dCat}}^k_{(\mathcal{B}^*,\mathcal{B})}$. 
\end{obs}



\section{Examples}

\noindent In this section we provide examples of double categories obtained through the methods provided by theorem \ref{thmmain}. We provide non-equivalent solutions to problem \ref{prob} in certain cases, we provide interpretations whithin the theory of double categories of classic constructions in the theory of tensor categories and we provide a linear double category of von Neumann algebras.

\

\noindent \textit{Single object decorated bicategories}

\

\noindent Let $M$ be a monoid. We write $\Omega M$ for the delooping category of $M$. By the Eckman-Hilton argument \cite{EckmanHilton} $\Omega M$ admits a strict monoidal structure only in the case in which $M$ is commutative, in which case the tensor product operation on $\Omega M$ is provided by the product operation of $M$. Given a monoidal category $D$ the pair $(\Omega M,2D)$ is a decorated bicategory and every strict single object decorated bicategory is of this form. We consider double categories of the form $C^\Phi$ for functors $\Phi$ in $Fun(\Omega M,\mbox{\textbf{Cat}}^\otimes)_{2D}$ for a monoidal category $D$. We first consider the case in which $D$ of the form $\Omega N$ for a commutative monoid $N$. 

Let $\Phi$ be a functor in $Fun(\Omega M,\mbox{\textbf{Cat}}^\otimes)_{2\Omega N}$ for monoids $M,N$ where $N$ is commutative. Any such $\Phi$ associates $End_{2\Omega N}(\ast)$ to the only object $\ast$ in $\Omega M$ and associates a monoidal endofunctor $\Phi_m$ of $End_{2\Omega N}(\ast)$ to every $m\in M$. That is, every functor $\Phi$ in $Fun(\Omega M,\mbox{\textbf{Cat}}^\otimes)_{2\Omega N}$ associates $\Omega N$ to the only object $\ast$ of $\Omega M$ and defines a morphism, which we keep denoting by $\Phi$, from $M$ to the monoid of monoid endomorphisms $End(N)$ of $N$. The category of objects $C^\Phi_0$ of $C^\Phi$ is equal to $\Omega M$. To compute $C^\Phi$ we thus only need to compute the category of morphisms $C^\Phi_1$ of $C^\Phi$. The category $\tilde{2\Omega N}_2$ is empty and thus $C^\Phi_1$ is equal to $\int_{\Omega M}\Phi$ which in this case is equal to the delooping category $\Omega(N\rtimes_\Phi M)$. The squares of $C^\Phi$ are thus all of the form:

\begin{center}
\begin{tikzpicture}
  \matrix (m) [matrix of math nodes,row sep=3em,column sep=3em,minimum width=2em]
  {
     \ast&\ast \\
     \ast&\ast \\};
  \path[-stealth]
    (m-1-1) edge [red] node [above]{}(m-1-2)
            edge [white] node [black][fill=white]{$(m,n)$}(m-2-2)    
    (m-2-1) edge [red] node {} (m-2-2)
    (m-1-1) edge node [left]{$m$} (m-2-1)
    (m-1-2) edge node [right]{$m$} (m-2-2);
\end{tikzpicture}

\end{center}

\noindent with $m\in M,n\in N$. Vertical and horizontal composition of such squares is defined by composition in $N\rtimes_\Phi M$ and by the product operation in $M$ respectively. Analogous computations apply in the case in which either $M$ or $N$ is a group and in the case $M$ and $N$ are algebras over a field $k$ or von Neumann algebras. We have the following corollary of theorem \ref{thmmain}.

\begin{cor}\label{1stcorollary}
Let $M,N$ be monoids such that $N$ is commutative. Problem \ref{prob} for the decorated bicategory $(\Omega M,2\Omega N)$ admits at least as many solutions, up to double equivalence, as there are isomorphism types of extensions $N\rtimes_\Phi N$.
\end{cor}

\begin{proof}
The result follows by theorem \ref{thmmain}, by the observations above and by the easy observation that any double equivalence $F:C^\Phi\to C^{\Psi}$ for $\Phi,\Psi$ in $Fun(\Omega M,\mbox{\textbf{Cat}}^\otimes)_{2\Omega N}$ induces an isomorphism from $N\rtimes_\Phi M$ to $N\rtimes _\Psi M$. This concludes the proof of the lemma.
\end{proof}

\noindent Corollary \ref{1stcorollary} can be refined as follows: Suppose that $M,N$ are groups. In that case observe that both the category of objects and the category of morphisms of the double cateogry $C^\Phi$ presented in the proof of corollary \ref{1stcorollary} are groupoids. In that case the decorated bicategory $(\Omega M,2\Omega N)$ not only admits solutions but admits solutions internal to the category of groupoids and groupoid morphisms. A similar comment applies in the case in which $M,N$ are algebras over a field $k$. In that case the double category $(\Omega M,2\Omega N)$ not only admits internalizations but admits $k$-linear internalizations. We study a stronger version of corollary \ref{1stcorollary} in section 6. A non-strict case of the linear version of corollary \ref{1stcorollary} would imply a study of extensions of $E_2$ algebras.

\

\noindent \textit{Graded vector spaces}

\

\noindent Let $G,H$ be groups. Suppose $H$ is abelian. We write $\mathcal{C}_G(H)$ for the monoidal category whose objects are of the form $\delta_g$ with $g\in G$ and such that Hom$_{\mathcal{C}_G(H)}(\delta_g,\delta_{g'})$ is equal to $H$ when $g=g'$ and empty otherwise. Tensor product in $\mathcal{C}_G(H)$ is defined by $\delta_g\otimes \delta_{g'}=\delta_{gg'}$ on objects and by $a\otimes a'=aa'$ for endomorphisms $a,a'$ of $\delta_g$ and $\delta_{g'}$. The pair $(\Omega G,2\mathcal{C}_G(H))$ is a decorated 2-category.

Let $\Phi$ be a functor in $Fun(\Omega G,\mbox{Cat}^\otimes)_{2\mathcal{C}_G(H)}$. Such a functor is such that $\Phi(\ast)=\mathcal{C}_G(H)$ and such that $\Phi_g$ is a monoidal endofunctor of $\mathcal{C}_G(H)$ for every $g\in G$. Assume that $\Phi_g$ is the identity on the collection of objects of $\mathcal{C}_G(H)$ and that the restriction of $\Phi_g$ to endomorphisms of $\delta_g,\delta_{g'}$ is equal for every $g,g'\in G$. Such a functor $\Phi$ defines a morphism, which we keep denoting by $\Phi$, from $G$ to $Aut(H)$. In this case the category of objects $C^\Phi_0$ of $C^\Phi$ is equal to $\Omega G$, the category $\tilde{2\mathcal{C}_G(H)}$ is empty and thus the category of morphisms $C^\Phi_1$ of $C^\Phi$ is equal to $\int_{\Omega G}\Phi$. The collection of objects of $\int_{\Omega G}\Phi$ is equal to $G$. Given $g,g'\in G$ the collection of morphisms Hom$_{\int_{\Omega G}\Phi}(g,g')$ is the collection of all pairs $(g,h)$ with $g\in G$ and $h\in H$ whenever $g=g'$ and is empty otherwise. That is, the squares of $C^\Phi$ are all of the form:

\begin{center}
\begin{tikzpicture}
  \matrix (m) [matrix of math nodes,row sep=3em,column sep=3em,minimum width=2em]
  {
     \ast&\ast \\
     \ast&\ast \\};
  \path[-stealth]
    (m-1-1) edge node [above]{$\delta_{g'}$}(m-1-2)
            edge [white] node [black][fill=white]{$(g,h)$}(m-2-2)    
    (m-2-1) edge node [below]{$\delta_{g'}$} (m-2-2)
    (m-1-1) edge node [left]{$g$} (m-2-1)
    (m-1-2) edge node [right]{$g$} (m-2-2);
\end{tikzpicture}

\end{center}

\noindent where $g,g'\in G$ and $h\in H$. The vertical composition of any two such pairs $(g',h')\boxvert(g,h)$ is equal to $(g'g,h'\Phi_{g'}(h))$ and the horizontal composition is given by $(g,h)\boxminus (g,h')=(g,hh')$. Identifying every $g\in G$ with the object $\delta_g$ in $\mathcal{C}_G(H)$ and identifying every pair $(g,h)$ with $g\in G$ and $h\in H$ with $h$ as an endomorphism of $\delta_g$ it is easily seen that the vertical monoidal category of $C^\Phi$, i.e. the category whose objects are vertical morphisms of $C^\Phi$, whose morphisms are 2-morphisms in $C^\Phi$, whose composition operation is the horizontal composition of $C^\Phi$ and whose tensor product is the vertical composition of $C^\Phi$, is isomorphic to the category $\mathcal{C}_G(H,\Phi)$ obtained from $\mathcal{C}_G(H)$ by twisting by $\Phi$, see \cite{EGNO}. The obvious linear extension of this example also holds. That is, if $k$ is a field, $G,H$ are groups, $H$ is abelian and $\Phi$ is a functor in $Fun(G,\mbox{Cat}^\otimes_k)_{2\mbox{\textbf{Vec}}_k^G(H)}$ then $C^\Phi$ is $k$-linear and its vertical category is isomorphic, as $k$-linear tensor categories to \textbf{Vec}$_k^G(H,\Phi)$.

\

\noindent \textit{Algebras and von Neumann algebras}

\

\noindent Let $k$ be a field. We write $\underline{\mbox{\textbf{Alg}}}_k$ for the bicategory with unital $k$-algebras, representations, and intertwining operators as 0-, 1-, and 2-dimensional cells. The horizontal composition operation $\circledast$ on $\underline{\mbox{\textbf{Alg}}}_k$ is the relative tensor product of representations. The tensor category of endomorphisms $End_{\underline{\mbox{\textbf{Alg}}}_k}(A)$ of an algebra $A$ is the category of $A$-bimodules $Mod_A$. We write \textbf{Alg}$_k$ for the category whose objects are $k$-algebras and whose morphisms are unital algebra morphisms. The pair $(\mbox{\textbf{Alg}}_k,\underline{\mbox{\textbf{Alg}}}_k)$ is a linear decorated bicategory. Let $\Phi$ be the functor in $Fun(\mbox{\textbf{Alg}}_k,\mbox{\textbf{Cat}}^\otimes)_{\underline{\mbox{\textbf{Alg}}}_k}$ such that for every $f:A\to B$ the functor $\Phi_f$ is the constant tensor functor on $_BB_B$. The linear double category $C^\Phi$ is such that $H^*C^\Phi=(\mbox{\textbf{Alg}}_k,\underline{\mbox{\textbf{Alg}}}_k)$. Non-globular squares 

\begin{center}
\begin{tikzpicture}
  \matrix (m) [matrix of math nodes,row sep=3em,column sep=3em,minimum width=2em]
  {
     A&A \\
     B&B \\};
  \path[-stealth]
    (m-1-1) edge node [above]{$M$}(m-1-2)
            edge [white] node [black][fill=white]{$\varphi$}(m-2-2)    
    (m-2-1) edge node [below]{$N$} (m-2-2)
    (m-1-1) edge node [left]{$f$} (m-2-1)
    (m-1-2) edge node [right]{$f$} (m-2-2);
\end{tikzpicture}

\end{center}

\noindent in $C^\Phi$ are formed by pairs $(f,\varphi)$ where $f:A\to B$ and $\varphi$ is a morphism of $B$-modules from the trivial bimodule $B$ to $N$.

Analogous constructions apply for any flavor of algebra, representation and intertwiner operator. Of special relevance is the case of von Neumann algebras, Hilbert bimodules, and bounded intertwiner operators. We write $\underline{[W^*]}$ for the bicategory whose 0-,1- and 2-cells are von Neumann algebras, Hilbert bimodules, and intertwiner operators respectively. The horizontal composition operation $\boxminus$ is provided by Connes fusion $\boxtimes$ and the horizontal identity $i_A$ of a von Neumann algebra $A$ is provided by the Haagerup standard form $L^2A$ of $A$, see \cite{Thom,Landsman}. Given a von Neumann algebra $A$ the category of endomorphisms $End_{\underline{[W^*]}}(A)$ is the category of Hilbert bimodules $Mod_A$. We write \textbf{vN} for the category whose objects are von Neumann algebras and their morphisms. The pair $(\mbox{\textbf{vN}},\underline{[W^*]})$ is a linear decorated bicategory. Wite $\Phi$ for the functor in $Fun(\mbox{\textbf{vN}},\mbox{\textbf{Cat}}^\otimes)_{\underline{[W^*]}}$ such that $\Phi_f$ is the constant functor on $L^2B$ for every von Neumann algebra morphism $f:A\to B$. The double category $C^\Phi$ is such that $H^*C^\Phi=(\mbox{\textbf{vN}},\underline{[W^*]})$. A non-globular square

\begin{center}
\begin{tikzpicture}
  \matrix (m) [matrix of math nodes,row sep=3em,column sep=3em,minimum width=2em]
  {
     A&A \\
     B&B \\};
  \path[-stealth]
    (m-1-1) edge node [above]{$M$}(m-1-2)
            edge [white] node [black][fill=white]{$\varphi$}(m-2-2)    
    (m-2-1) edge node [below]{$N$} (m-2-2)
    (m-1-1) edge node [left]{$f$} (m-2-1)
    (m-1-2) edge node [right]{$f$} (m-2-2);
\end{tikzpicture}

\end{center}

\noindent in $C^\Phi$ is now formed by a pair $(f,\varphi)$ where $f:A\to B$ is a morphism of von Neumann algebras and where $\varphi:L^2B\to N$ is an intertwiner operator of $B$-bimodules. The horizontal identity functor $i$ in $C^\Phi$ associates to every morphism of von Neumann algebras $f:A\to B$ the pair $(f,id_{L^2B})$ and the horizontal composition functor $\boxminus$ in $C^\Phi$ is defined as $(f,\varphi)\boxminus(f,\psi)=(f,\varphi\boxtimes \psi)$ for every horizontally compatible pair $(f,\varphi),(f,\psi)$ in $C^\Phi$.

\

\noindent \textit{Induction}

\

\noindent The examples provided above are rather unsatisfactory solutions to problem \ref{prob} for the decorated bicategories of algebras and von Neumann algebras. We present an alternative solution for proper subategories of the corresponding decorations.

Let $A,B$ be $k$-algebras. Let $f:A\to B$ be a morphism. We say that $f$ is tensor inductive if the induced representation functor $f_*:Mod_A\to Mod_B$ associated to $f$ is a tensor functor. Examples of tensor inductive morphisms are isomorphism and morphisms between commutative algebras. The composition of tensor unductive morphisms is again tensor inductive and the identity $id_A$ of an algebra $A$ is always tensor inductive. We write \textbf{Alg}$^{TI}_k$ for the category whose objects are unital $k$-algebras and whose morphisms are unital tensor inductive morphisms. Thus defined \textbf{Alg}$^{TI}_k$ has the category of commutative algebras and the underlying groupoid of \textbf{Alg}$_k$ as subcategories. The pair $(\mbox{\textbf{Alg}}^{TI}_k,\underline{\mbox{\textbf{Alg}}}_k)$ is a linear decorated bicaegory.

Write $\Phi:\mbox{\textbf{Alg}}_k^{TI}\to\mbox{\textbf{Cat}}^\otimes$ for the functor such that $\Phi(A)=Mod_A$ for every algebra $A$ and such that $\Phi(f)=f_*$ for every morphism $f:A\to B$ in \textbf{Alg}$^{TI}_k$. Thus defined $\Phi$ is an object in $Fun(\mbox{\textbf{Alg}}^{TI}_k,\mbox{\textbf{Cat}}^\otimes)_{\underline{\mbox{\textbf{Alg}}}_k}$. The double category $C^\Phi$ is a linear solution to the equation $H^*C^\Phi=(\mbox{\textbf{Alg}}^{TI}_k,\underline{\mbox{\textbf{Alg}}}_k)$. A non-globular square in $C^\Phi$ of the form

\begin{center}
\begin{tikzpicture}
  \matrix (m) [matrix of math nodes,row sep=3em,column sep=3em,minimum width=2em]
  {
     A&A \\
     B&B \\};
  \path[-stealth]
    (m-1-1) edge node [above]{$M$}(m-1-2)
            edge [white] node [black][fill=white]{$\varphi$}(m-2-2)    
    (m-2-1) edge node [below]{$N$} (m-2-2)
    (m-1-1) edge node [left]{$f$} (m-2-1)
    (m-1-2) edge node [right]{$f$} (m-2-2);
\end{tikzpicture}

\end{center}

\noindent is now formed by a pair $(f,\varphi)$ where $f$ is a tensor inductive morphism from $A$ to $B$, and where $\varphi$ is a morphism from $f_*M$ to $N$ in $Mod_B$. The horizontal identity $i_f$ of a tensor inductive morphism $f:A\to B$ is the pair $(f,id_B)$. A contravariant version of the above construction applies for morphisms $f$ such that the restriction functor $f^*$ is a tensor functor.

An analogous construction applies in the case of von Neumann algebras. We say that a morphism of von Neumann algebras $f:A\to B$ is $\boxtimes$-inductive if $f_*:Mod_A\to Mod_B$ is a tensor functor. We write  \textbf{vN}$^{\boxtimes}$ for the category of von Neumann algebras and $\boxtimes$-inductive morphisms. Again the pair $(\mbox{\textbf{vN}}^\boxtimes,\underline{[W^*]})$ is a decorated bicategory. Write $\Phi$ for the functor from \textbf{vN}$^\boxtimes$ to \textbf{Cat}$^\otimes$ such that $\Phi(A)=Mod_A$ for every $A$ and such that $\Phi_f=f_*$ for every morphism $f:A\to B$ in \textbf{vN}$^\boxtimes$. Thus defined $\Phi$ is a functor in $Fun(\mbox{\textbf{vN}}^\boxtimes,\mbox{\textbf{Cat}}^\otimes)_{[W^*]}$. The double category $C^\Phi$ is a linear double category satisfying the equation $H^*C^\Phi=(\mbox{\textbf{vN}}^\boxtimes,{\underline{[W^*]}})$. A non-globular square

\begin{center}
\begin{tikzpicture}
  \matrix (m) [matrix of math nodes,row sep=3em,column sep=3em,minimum width=2em]
  {
     A&A \\
     B&B \\};
  \path[-stealth]
    (m-1-1) edge node [above]{$M$}(m-1-2)
            edge [white] node [black][fill=white]{$\varphi$}(m-2-2)    
    (m-2-1) edge node [below]{$N$} (m-2-2)
    (m-1-1) edge node [left]{$f$} (m-2-1)
    (m-1-2) edge node [right]{$f$} (m-2-2);
\end{tikzpicture}

\end{center}

\noindent in $C^\Phi$ is now defined by a pair $(f,\varphi)$ where $f:A\to B$ is a $\boxtimes$-inductive morphism and where $\varphi$ is an intertwining operator from $f_*M$ to $N$ in $Mod_B$. The horizontal identity $i_f$ of a von Neumann algebra morphism $f:A\to B$ is the pair $(f,id_{L^2B})$. The horizontal composition functor $\boxminus$ of $C^\Phi$ again acts on pairs of morphisms $(f,\varphi),(f,\psi)$ as $(f,\varphi)\boxminus(f,\psi)=(f,\varphi\boxtimes\psi)$. Compare this to \cite{Bartels1}. A similar construction applies when substituting induction by restriction and when substituting von Neumann algebras with $C^*$-algebras.

\section{Foldings and cofoldings}

\noindent In this section we analyze possible folding and cofolding structures on certain double categories constructed through the methods of theorem \ref{thmmain}. We refer the reader to \cite{FioreFreeMonads} for a treatment of holonomies, foldings, and cofoldings on double categories. 

\

\noindent We consider the following situation: Let $M,A$ be monoids. Suppose $A$ is commutative. In that case the pair $(\Omega M,2\Omega A)$ is a decorated 2-category. Let $\Phi$ be a functor in $Fun(\Omega M,\mbox{\textbf{Cat}}^\otimes)_{2\Omega}$. A holonomy on $C^\Phi$ is a 2-functor from $C^\Phi_0$ to $H2\Omega A$, i.e. a holonomy on $C^\Phi$ is a morphism from $M$ to the category of horizontal morphisms of $C^\Phi$ which are the identity on objects. This category is trivial since $\Omega A$ has a single object. Likewise $C^\Phi$ admits a single trivial coholonomy. Foldings are extensions of holonomies establishing compatible bijections between certain squares. Cofoldings are defined similarly. We prove the following proposition.

\begin{prop}\label{fold1}
Let $M,A$ be monoids. Suppose $A$ is commutative. Let $\Phi$ be a functor in $Fun(\Omega M,\mbox{\textbf{Cat}}^\otimes)_{2\Omega A}$ such that for $\Phi_m=id_{\Omega A}$ for every $m\in M$. In that case $C^\Phi$ admits a folding and a cofolding.
\end{prop}

\begin{proof}
Let $M,A$ be monoids. Suppose that $A$ is commutative. Let $\Phi$ be a functor in $Fun(\Omega M,\mbox{\textbf{Cat}}^\otimes)_{2\Omega A}$, such that $\Phi_m=id_{\Omega A}$ for every $m\in M$. We wish to prove in this case that the trivial holonomy on $C^\Phi$ extends to a folding on $C^\Phi$.

We wish to prove that the trivial holonomy on $C^\Phi$ extends to a 2-functor $\Lambda:C^\Phi_0\to \mbox{\textbf{Q}}2\Omega A$ such that $H\Lambda=id_{2\Omega A}$ and such that $\Lambda$ is fully faithful on squares. For every square boundary:

\begin{center}
\begin{tikzpicture}
  \matrix (m) [matrix of math nodes,row sep=4em,column sep=4em,minimum width=2em]
  {
     \ast&\ast\\ 
     \ast&\ast\\};
  \path[-stealth]
    (m-1-1) edge [red]node [above] {} (m-1-2)
            edge node [left] {$m$}(m-2-1)
    (m-2-1) edge [red]node [below]{} (m-2-2)
    (m-1-2) edge node [right]{$m$}(m-2-2);
\end{tikzpicture}

\end{center}

\noindent in $C^\Phi$ we construct a bijection $\Lambda_{m,1}^{1,m}$ between the collection of squares, in $C^\Phi$, of the following two forms

\begin{center}
\begin{tikzpicture}
  \matrix (m) [matrix of math nodes,row sep=4em,column sep=4em,minimum width=2em]
  {
     \ast&\ast&\ast&\ast\\ 
     \ast&\ast&\ast&\ast\\};
  \path[-stealth]
    (m-1-1) edge [red]node [above] {} (m-1-2)
            edge node [left] {$m$}(m-2-1)
    (m-2-1) edge [red]node [below]{} (m-2-2)
    (m-1-2) edge node [right]{$m$}(m-2-2)
    (m-1-1) edge [white] node [black][fill=white]{$(m,a)$}(m-2-2)

    (m-1-3) edge [red]node {} (m-1-4)
            edge [blue] node [black][left] {$1$}(m-2-3)
    (m-2-3) edge [red]node {} (m-2-4)
    (m-1-4) edge [blue] node [black][right] {$1$}(m-2-4)
    (m-1-3) edge [white] node [black][fill=white]{$(1,a)$}(m-2-4);
\end{tikzpicture}

\end{center}

\noindent such that the collection $\Lambda$ of these bijection forms the desired 2-functor. We define $\Lambda_{m,1}^{1,m}$ as the function associating a 2-morphism $(m,a)$ of the form on the left-hand side above, the square $(1,a)$. Thus defined $\Lambda_{m,1}^{1,m}$ is clearly a bijection from the set of diagrams of the form on the left-hand side above to the set of diagrams of the form on the right-hand side above. It is easily seen that the collection $\Lambda$ of bijections of the form $\Lambda_{m,1}^{1,m}$, with $m\in M$, is the identity on globular 2-morphisms of $C^\Phi$ and that is compatible with identities and horizontal composition in $C^\Phi$. We prove that $\Lambda$ is compatible with vertical composition in $C^\Phi$. To this end consider a vertical composition of squares of the form:

\begin{center}

\begin{tikzpicture}
  \matrix (m) [matrix of math nodes,row sep=4em,column sep=4em,minimum width=2em]
  {
     \ast&\ast \\
     \ast&\ast \\
     \ast&\ast \\};
  \path[-stealth]
   (m-1-1) edge [red] node {} (m-1-2)
            edge node [left]{$m$}(m-2-1)
    (m-2-1) edge [red] node {} (m-2-2)
    (m-1-2) edge node [right]{$m$}(m-2-2)
    (m-1-1) edge [white] node [black][fill=white]{$(m,a)$}(m-2-2)
    (m-2-1) edge node [left] {$m'$}(m-3-1)
    (m-2-2) edge node [right]{$m'$}(m-3-2)
    (m-2-1) edge [white] node [black][fill=white]{$(m',a')$}(m-3-2)
    (m-3-1) edge [red] node {} (m-3-2);
    
\end{tikzpicture}
\end{center}

\noindent by the way $\Phi$ was chosen, the bove diagram is equal to the square:

\begin{center}

\begin{tikzpicture}
  \matrix (m) [matrix of math nodes,row sep=4em,column sep=4em,minimum width=2em]
  {
     \ast&\ast \\
     \ast&\ast \\};
  \path[-stealth]
   (m-1-1) edge [red] node {} (m-1-2)
            edge node [left]{$m'm$}(m-2-1)
    (m-2-1) edge [red] node {} (m-2-2)
    (m-1-2) edge node [right]{$m'm$}(m-2-2)
    (m-1-1) edge [white] node [black][fill=white]{$(m'm,a'a)$}(m-2-2);
    
\end{tikzpicture}
\end{center}

\noindent Applying $\Lambda$ to the above square we obtain the square:

\begin{center}

\begin{tikzpicture}
  \matrix (m) [matrix of math nodes,row sep=4em,column sep=4em,minimum width=2em]
  {
     \ast&\ast \\
     \ast&\ast \\};
  \path[-stealth]
   (m-1-1) edge [red] node {} (m-1-2)
            edge [blue] node [black][left]{$1$}(m-2-1)
    (m-2-1) edge [red] node {} (m-2-2)
    (m-1-2) edge [blue] node [black][right]{$1$}(m-2-2)
    (m-1-1) edge [white] node [black][fill=white]{$(1,a'a)$}(m-2-2);
    
\end{tikzpicture}
\end{center}

\noindent which in turn is equal to the square:

\begin{center}

\begin{tikzpicture}
  \matrix (m) [matrix of math nodes,row sep=4em,column sep=4em,minimum width=2em]
  {
     \ast&\ast&\ast \\
     \ast&\ast&\ast \\
     \ast&\ast&\ast\\};
  \path[-stealth]
   (m-1-1) edge [red] node {} (m-1-2)
            edge [blue]node {}(m-2-1)
    (m-2-1) edge [red] node {} (m-2-2)
    (m-1-2) edge [blue] node {}(m-2-2)
    (m-1-1) edge [white] node [black][fill=white]{$(1,a)$}(m-2-2)
    (m-2-1) edge [blue] node {}(m-3-1)
    (m-2-2) edge [blue]node {}(m-3-2)
    (m-2-1) edge [white] node [black][fill=white]{$(1,1)$}(m-3-2)
    (m-3-1) edge [red] node {} (m-3-2)
    
    (m-1-2) edge [red] node {} (m-1-3)
    (m-2-2) edge [red] node {} (m-2-3)
    (m-1-3) edge [blue] node {}(m-2-3)
    (m-1-2) edge [white] node [black][fill=white]{$(1,1)$}(m-2-3)
    (m-2-3)edge [blue]node {}(m-3-3)
    (m-3-2)edge[red] node {}(m-3-3)
    (m-2-2) edge [white] node [black][fill=white]{$(1,a')$}(m-3-3);
    
\end{tikzpicture}
\end{center}

\noindent We conclude that $\Lambda$, as defined above, is compatible with vertical composition in $C^\Phi$ and thus that $\Lambda$ is a folding on $C^\Phi$. a similar argument proves that the trivial coholonomy on $C^\Phi$ extends to a cofolding on $C^\Phi$. This concludes the proof.

\end{proof}

\noindent From proposition \ref{fold1} and from \cite{SchulmanFramed} we have the following corollary.

\begin{cor}\label{framed1}
Let $M,A$ be monoids. Suppose $A$ is commutative. Let $\Phi$ be a functor in $Fun(\Omega M,\mbox{\textbf{Cat}}^\otimes)_{2\Omega A}$ such that $\Phi_m=id_{\Omega A}$ for every $m\in M$. In that case $C^\Phi$ is a framed 2-category.
\end{cor}

\noindent The following is an example of a double category of the form $C^\Phi$ with $M,D$ as in propostion \ref{fold1} but with non-identity $\Phi$ such that $C^\Phi$ does not admit a folding or a cofolding and is thus not a framed bicategory.

\begin{ex}
In the notation of proposition \ref{fold1} let $M=\mathbb{Z}_2$, let $D=\Omega \mathbb{Z}_3$, and let $\Phi$ be the functor in $Fun(\Omega\mathbb{Z}_2,\mbox{\textbf{Cat}}^\otimes)_{2\Omega\mathbb{Z}_3}$ such that $\Phi_a(m)=am$ for every $a\in\mathbb{Z}_2,m\in\mathbb{Z}_3$ where we write $\mathbb{Z}_2$ multiplicatively and $\mathbb{Z}_3$ additively. The double category $C^\Phi$ has a single horizontal morphism with endomorphism group isomorphic to $D_6$. We prove that $C^\Phi$ does not admit foldings. Suppose $\Lambda$ is a folding of $C^\Phi$. $\Lambda$ establishes a bijection between squares of the following two forms:

\begin{center}
\begin{tikzpicture}
  \matrix (m) [matrix of math nodes,row sep=4em,column sep=4em,minimum width=2em]
  {
     \ast&\ast&\ast&\ast\\ 
     \ast&\ast&\ast&\ast\\};
  \path[-stealth]
    (m-1-1) edge [red] node {} (m-1-2)
            edge node [left] {$-1$}(m-2-1)
    (m-2-1) edge [red] node {} (m-2-2)
    (m-1-2) edge node [right]{$-1$}(m-2-2)

    (m-1-3) edge [red] node [above]{} (m-1-4)
            edge [blue] node [black][left] {$1$}(m-2-3)
    (m-2-3) edge [red] node [below] {} (m-2-4)
    (m-1-4) edge [blue] node [black][right] {$1$}(m-2-4);
\end{tikzpicture}

\end{center}

\noindent Moreover $\Lambda$ is compatible with vertical and horizontal composition in $C^\Phi$. Compatibility of $\Lambda$ with horizontal composition says that $\varphi:\mathbb{Z}_3\to\mathbb{Z}_3$ such that $\Lambda(-1,a)=(1,\varphi(a))$ for every $a\in\mathbb{Z}_3$ is such that $\varphi\in Aut(\mathbb{Z}_3)= \mathbb{Z}_2$. Now, the compatibility of $\Lambda$ with vertical composition in $C^\Phi$ says that $\Lambda(-1,a)\boxminus\Lambda(-1,a')=\Lambda(1,a'-a)$ for every $a,a'\in\mathbb{Z}_3$, but the left hand side of the above equation is equal to $(1,\varphi(a)+\varphi(a'))$. Thus $\varphi$ is an automorphism of $\mathbb{Z}_3$ satisfying the equation $\varphi(a)+\varphi(a')=a-a'$ for every $a,a'\in\mathbb{Z}_3$. Such $\varphi$ does not exist and thus $C^\Phi$ does not admit foldings.
\end{ex}

\section{Vertical length}

\noindent In this section we articulate the idea that relations between horizontal and vertical compositions of globular and horizontal identity squares in double categories constructed through the methods of theorem \ref{thmmain} should be relatively simple to understand. We do this through the notion of vertical length. We refer the reader to \cite{yo1} for the precise definitions.

\

\noindent For the convenience of the reader we rephrase the definition of double categories of vertical length 1 in a way specific to the pourpuses of this section. Given a double category $C$ we write $V^1_{\gamma C}$ for the first vertical category of $C$, i.e. $V^1_{\gamma C}$ denotes the subcategory of $C_1$ generated by the globular and horizontal identity squares of $C$. The second vertical category $V^2_{\gamma C}$ of $C$ is the subcategory of $C_1$ generated by horizontal compositions of morphisms in $V^1_{\gamma C}$. We say that a double category $C$ has vertical length 1, $\ell C=1$ in symbols, if $V^1_{\gamma C}=V^2_{\gamma C}$ or equivalently if $V^1_{\gamma C}$ is closed under the operation of taking horizontal compositions. Intuitively a double category $C$ has vertical length 1 whenever given two squares of the form:

\begin{center}
\begin{tikzpicture}
  \matrix (m) [matrix of math nodes,row sep=4em,column sep=4em,minimum width=2em]
  {
     a&a&a&a\\ 
     b&b&b&b\\};
  \path[-stealth]
    (m-1-1) edge node [above] {$\alpha$} (m-1-2)
            edge node [left] {$f$}(m-2-1)
    (m-2-1) edge node [below]{$\beta$} (m-2-2)
    (m-1-2) edge node [right]{$f$}(m-2-2)
    (m-1-1) edge [white] node [black][fill=white]{$\varphi$}(m-2-2)

    (m-1-3) edge node [above]{$\alpha'$} (m-1-4)
            edge node [left] {$f$}(m-2-3)
    (m-2-3) edge node [below] {$\beta'$} (m-2-4)
    (m-1-4) edge node [right]{$f$}(m-2-4)
    (m-1-3) edge [white] node [black][fill=white]{$\psi$}(m-2-4);
\end{tikzpicture}

\end{center}

\noindent in $C$ and factorizations of $\varphi$ and $\psi$ as vertical compositions of squares of the form:

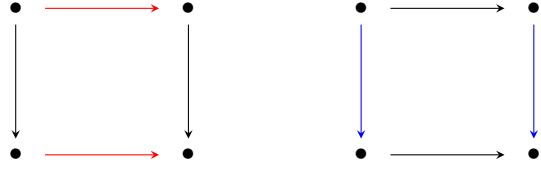
\begin{center}
\begin{tikzpicture}
  \matrix (m) [matrix of math nodes,row sep=4em,column sep=4em,minimum width=2em]
  {
     \bullet&\bullet&\bullet&\bullet\\ 
     \bullet&\bullet&\bullet&\bullet\\};
  \path[-stealth]
    (m-1-1) edge [red] node [above] {} (m-1-2)
            edge node {}(m-2-1)
    (m-2-1) edge [red] node {} (m-2-2)
    (m-1-2) edge node {}(m-2-2)

    (m-1-3) edge node {} (m-1-4)
            edge [blue] node {}(m-2-3)
    (m-2-3) edge node {} (m-2-4)
    (m-1-4) edge [blue] node {}(m-2-4);
\end{tikzpicture}

\end{center}

\noindent the terms of these factorizations of $\varphi$ and $\psi$ can be re-arranged in such a way that $\varphi$ and $\psi$ can be composed horizontally by composing terms of the corresponding factorizations horizontally one by one. An explicit description of 2-morphisms in $V^1_{\gamma C}$ for any double category $C$ is provided in \cite{yo1}. The following lemma uses this result to provide an explicit description of squares in $V^1_{\gamma C}$ for every double category $C$.

\begin{lem}\label{lemmaimportant}
Let $(\mathcal{B}^*,\mathcal{B})$ be a decorated bicategory. Let $\Phi$ be an object of $Fun(\mathcal{B}^*,\mbox{\textbf{Cat}}^\otimes)_\mathcal{B}$. Let $(f,\varphi)$ be a morphism in $\int_{\mathcal{B}^*}\Phi$ with domain and codomain, in $\mathcal{B}$, the 1-cells $\alpha,\beta$ of 0-cells $a,b$ in $\mathcal{B}$ respectively. The following three conditions for $(f,\varphi)$ are equivalent:

\begin{enumerate}
    \item $(f,\varphi)$ is a morphism in $V^1_{\gamma C^\Phi}$.
    \item $(f,\varphi)$ admits a factorization in $C^\Phi$ as:
    
    \begin{center}

\begin{tikzpicture}
  \matrix (m) [matrix of math nodes,row sep=3em,column sep=7em,minimum width=2em]
  {
     \alpha&\beta \\
     i_a&i_b \\};
  \path[-stealth]
    (m-1-1) edge node [left] {$(id_a,\psi)$}(m-2-1)
            edge node [above] {$(f,\varphi)$}(m-1-2)
    (m-2-1) edge node [below] {$(f,id_{i_a})$} (m-2-2)
    (m-2-2) edge node [right] {$(id_\beta,\eta)$} (m-1-2);
\end{tikzpicture}
\end{center}
    
    \noindent for 2-cells $\psi,\eta$ in $\mathcal{B}$ from $\alpha$ to $i_a$ and from $i_b$ to $\beta$ respectively. 
    
    \item $(f,\varphi)$ admits a pictorial representation in $C^\Phi$ as:
    
   
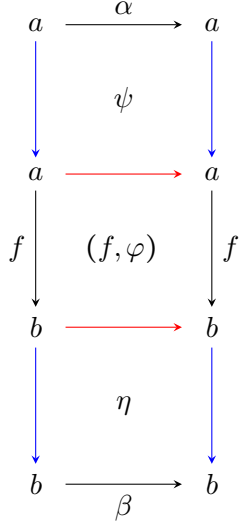
\begin{center}
\begin{tikzpicture}
  \matrix (m) [matrix of math nodes,row sep=4em,column sep=4em,minimum width=2em]
  {
     a&a \\
     a&a \\
     b&b \\
     b&b \\};
  \path[-stealth]
    (m-1-1) edge node [above] {$\alpha$} (m-1-2)
            edge [blue] node {}(m-2-1)
    (m-2-1) edge [red] node {} (m-2-2)
    (m-1-2) edge [blue] node {}(m-2-2)
    (m-1-1) edge [white] node [black][fill=white]{$\psi$}(m-2-2)
    (m-2-1) edge node [left]{$f$}(m-3-1)
    (m-2-2) edge node [right]{$f$}(m-3-2)
    (m-2-1) edge [white] node [black][fill=white]{$(f,\varphi)$}(m-3-2)
    (m-3-1) edge [red] node {} (m-3-2)
    (m-3-1) edge [blue] node {}(m-4-1)
    (m-3-2) edge [blue] node {}(m-4-2)
    (m-4-1) edge node [below] {$\beta$} (m-4-2)
    (m-3-1) edge [white] node [black][fill=white] {$\eta$}(m-4-2);
\end{tikzpicture}

\end{center}

    \noindent for 2-cells $\psi,\eta$ in $\mathcal{B}$ from $\alpha$ to $i_a$ and from $i_b$ to $\beta$ respectively.
\end{enumerate}
\end{lem}

\begin{proof}
Let $(\mathcal{B}^*,\mathcal{B})$ be a decorated bicategory. Let $\Phi$ be an object in the category $Fun(\mathcal{B}^*,\mbox{\textbf{Cat}}^\otimes)_\mathcal{B}$. Let $(f,\varphi)$ be a morphism in $\int_{\mathcal{B}^*}\Phi$ with domain and codomain the endomorphisms $\alpha,\beta$ of 0-cells $a,b$ in $\mathcal{B}$ respectively. We wish to prove that the three conditions in the statement of the lemma are equivalent.

Condition 3 is a pictorial interpretation of condition 2 above and these conditions are easily seen to be equivalent. We thus wish to prove that conditions 2 and 3 are equivalent to condition 1. Observe first that if $(f,\varphi)$ satisfies conditions 2 and 3 then $(f,\varphi)$ admits a representation as a vertical composition of globular squares of $C^\Phi$ and horizontal identities of vertical morphisms of $C^\Phi$. Thus, if $(f,\varphi)$ satisfies conditions 2 and 3 then $(f,\varphi)$ is a morphism in $V^1_{\gamma C^\Phi}$. We prove that if $(f,\varphi)$ is a morphism in $V^1_{\gamma C^\Phi}$ then $(f,\varphi)$ admits a pictorial representation as in condition 3 above.

Observe first that given objects $a,b$ in $\mathcal{B}^*$, an endomorphism $\psi$ of $i_a$ in $\mathcal{B}$ and a morphism $f:a\to b$ in $\mathcal{B}^*$ the two squares of $C^\Phi$ represented by the diagrams:

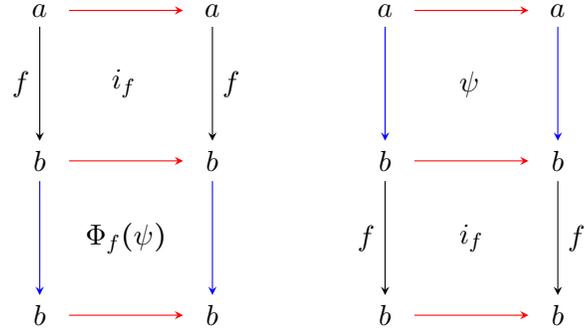
\begin{center}

\begin{tikzpicture}
  \matrix (m) [matrix of math nodes,row sep=4em,column sep=4em,minimum width=2em]
  {
     a&a&a&a \\
     b&b&b&b \\
     b&b&b&b\\};
  \path[-stealth]
   (m-1-1) edge [red] node {} (m-1-2)
            edge node [left]{$f$}(m-2-1)
    (m-2-1) edge [red] node {} (m-2-2)
    (m-1-2) edge node [right]{$f$}(m-2-2)
    (m-1-1) edge [white] node [black][fill=white]{$i_f$}(m-2-2)
    (m-2-1) edge [blue] node {}(m-3-1)
    (m-2-2) edge [blue]node {}(m-3-2)
    (m-2-1) edge [white] node [black][fill=white]{$\Phi_f(\psi)$}(m-3-2)
    (m-3-1) edge [red] node {} (m-3-2)
    
    (m-1-3) edge [red] node {} (m-1-4)
    (m-2-3) edge [red] node {} (m-2-4)
    (m-1-3) edge [blue] node {}(m-2-3)
    (m-1-4) edge [blue] node {}(m-2-4)
    (m-1-3) edge [white] node [black][fill=white]{$\psi$}(m-2-4)
    (m-2-3) edge node [left]{$f$}(m-3-3)
    (m-2-4)edge node [right]{$f$}(m-3-4)
    (m-3-3)edge[red] node {}(m-3-4)
    (m-2-3) edge [white] node [black][fill=white]{$i_f$}(m-3-4);
    
\end{tikzpicture}
\end{center}

\noindent are equal if and only if the equation 

\[(f,id_{i_a})\boxvert(id_a,\psi)=(id_b,\Phi_f(\psi))(f,id_{i_a})\]

\noindent holds. It is easily see that the left hand side and the right hand side of the above equation are both equal to $(f,\Phi_f(\psi))$. The two squares in $C^\Phi$ represented above are thus always equal. By iterating this equation we obtain the fact that for any 2-morphism of $C^\Phi$ admitting a factorization as a vertical composition of 2-morphisms admitting pictorial representations as:

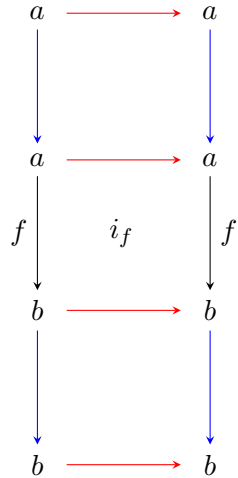
\begin{center}
\begin{tikzpicture}
  \matrix (m) [matrix of math nodes,row sep=4em,column sep=4em,minimum width=2em]
  {
     a&a \\
     a&a \\
     b&b \\
     b&b \\};
  \path[-stealth]
    (m-1-1) edge [red] node {} (m-1-2)
            edge [blue] node {}(m-2-1)
    (m-2-1) edge [red] node {} (m-2-2)
    (m-1-2) edge [blue] node {}(m-2-2)
    (m-2-1) edge node [left]{$f$}(m-3-1)
    (m-2-2) edge node [right]{$f$}(m-3-2)
    (m-2-1) edge [white] node [black][fill=white]{$i_f$}(m-3-2)
    (m-3-1) edge [red] node {} (m-3-2)
    (m-3-1) edge [blue] node {}(m-4-1)
    (m-3-2) edge [blue] node {}(m-4-2)
    (m-4-1) edge [red] node {} (m-4-2);
\end{tikzpicture}

\end{center}

\noindent admits a pictorial representation as:

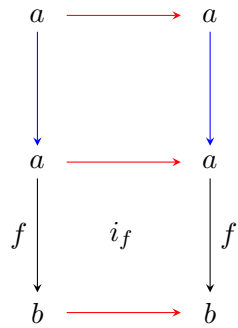
\begin{center}
\begin{tikzpicture}
  \matrix (m) [matrix of math nodes,row sep=4em,column sep=4em,minimum width=2em]
  {
     a&a \\
     a&a \\
     b&b \\};
  \path[-stealth]
    (m-1-1) edge [red] node {} (m-1-2)
            edge [blue] node {}(m-2-1)
    (m-2-1) edge [red] node {} (m-2-2)
    (m-1-2) edge [blue] node {}(m-2-2)
    (m-2-1) edge node [left]{$f$}(m-3-1)
    (m-2-2) edge node [right]{$f$}(m-3-2)
    (m-2-1) edge [white] node [black][fill=white]{$i_f$}(m-3-2)
    (m-3-1) edge [red] node {} (m-3-2);
\end{tikzpicture}

\end{center}

\noindent From this and from the obvious fact that every 2-morphism in $C^\Phi$ admitting a diagrammatic representation as:

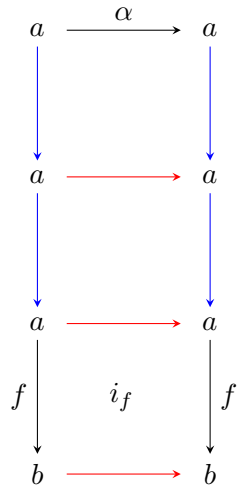
\begin{center}
\begin{tikzpicture}
  \matrix (m) [matrix of math nodes,row sep=4em,column sep=4em,minimum width=2em]
  {
     a&a \\
     a&a \\
     a&a \\
     b&b \\};
  \path[-stealth]
    (m-1-1) edge node [above]{$\alpha$} (m-1-2)
            edge [blue] node {}(m-2-1)
    (m-2-1) edge [red] node {} (m-2-2)
    (m-1-2) edge [blue] node {}(m-2-2)
    (m-2-1) edge [blue]node {}(m-3-1)
    (m-2-2) edge [blue]node {}(m-3-2)
    (m-3-1) edge [red] node {} (m-3-2)
    (m-3-1) edge node [left] {$f$}(m-4-1)
    (m-3-2) edge node [right] {$f$}(m-4-2)
    (m-4-1) edge [red] node {} (m-4-2)
    (m-3-1) edge [white] node [black][fill=white]{$i_f$}(m-4-2);
\end{tikzpicture}

\end{center}

\noindent admits a pictorial representation of the form:

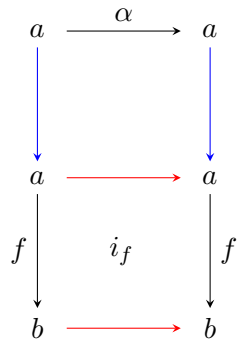
\begin{center}
\begin{tikzpicture}
  \matrix (m) [matrix of math nodes,row sep=4em,column sep=4em,minimum width=2em]
  {
     a&a \\
     a&a \\
     b&b \\};
  \path[-stealth]
    (m-1-1) edge node [above] {$\alpha$} (m-1-2)
            edge [blue] node {}(m-2-1)
    (m-2-1) edge [red] node {} (m-2-2)
    (m-1-2) edge [blue] node {}(m-2-2)
    (m-2-1) edge node [left]{$f$}(m-3-1)
    (m-2-2) edge node [right]{$f$}(m-3-2)
    (m-2-1) edge [white] node [black][fill=white]{$i_f$}(m-3-2)
    (m-3-1) edge [red] node {} (m-3-2);
\end{tikzpicture}

\end{center}

\noindent We conclude that whenever $(f,\varphi)$ is a morphism in $V^1_{\gamma C^\Phi}$ then $(f,\varphi)$ admits a presentation as in condition 3 in the proposition. This concludes the proof.

\end{proof}

\begin{prop}\label{vertlengthprop}
Let $(\mathcal{B}^*,\mathcal{B})$ be a decorated bicategory. Let $\Phi$ be a functor in $Fun(\mathcal{B}^*,\mbox{\textbf{Cat}}^\otimes)_{\mathcal{B}}$. In that case $\ell C^\Phi=1$.
\end{prop}

\begin{proof}
Let $a,b$ be objects in $C^\Phi$. Let $f:a\to b$ be a vertical morphism. Let $\alpha,\alpha'$ be horizontal endomorphisms of $a$ and let $\beta,\beta'$ be horizontal endomorphisms of $b$. Let $(f,\varphi)$ and $(f,\varphi')$ be squares in $V^1_{\gamma C^\Phi}$, from $\alpha$ to $\beta$ and from $\alpha'$ to $\beta'$ respectively. By lemma \ref{lemmaimportant} both $(f,\alpha)$ and $(f,\alpha')$ admit a pictorial representation as:

\begin{center}
\begin{tikzpicture}
  \matrix (m) [matrix of math nodes,row sep=4em,column sep=4em,minimum width=2em]
  {
     a&a&a&a\\ 
     a&a&a&a\\
     b&b&b&b\\ 
     b&b&b&b\\};
  \path[-stealth]
    (m-1-1) edge node [above]{$\alpha$} (m-1-2)
            edge [blue] node {}(m-2-1)
    (m-2-1) edge [red] node {} (m-2-2)
    (m-1-2) edge [blue] node {}(m-2-2)
    (m-1-1) edge [white] node [black][fill=white] {$\psi$} (m-2-2)
    (m-2-1) edge node [left]{$f$}(m-3-1)
    (m-2-2) edge node [right]{$f$}(m-3-2)
    (m-3-1) edge [red] node {} (m-3-2)
    (m-2-1) edge [white] node [black][fill=white]{$i_f$}(m-3-2)
    (m-3-1) edge [blue]node {}(m-4-1)
    (m-3-2) edge [blue]node {}(m-4-2)
    (m-4-1) edge node [below]{$\beta$} (m-4-2)
    (m-3-1) edge [white] node [black][fill=white]{$\eta$}(m-4-2)

    (m-1-3) edge node [above]{$\alpha'$} (m-1-4)
            edge [blue] node {}(m-2-3)
    (m-2-3) edge [red] node {} (m-2-4)
    (m-1-4) edge [blue] node {}(m-2-4)
    (m-1-3) edge [white] node [black][fill=white] {$\psi'$} (m-2-4)
    (m-2-3) edge node [left]{$f$}(m-3-3)
    (m-2-4) edge node [right]{$f$}(m-3-4)
    (m-3-3) edge [red] node {} (m-3-4)
    (m-2-3) edge [white] node [black][fill=white]{$i_f$}(m-3-4)
    (m-3-3) edge [blue]node {}(m-4-3)
    (m-3-4) edge [blue] node {}(m-4-4)
    (m-4-3) edge node [below]{$\beta'$} (m-4-4)
    (m-3-3) edge[white] node [black][fill=white]{$\eta'$}(m-4-4);
\end{tikzpicture}

\end{center}

\noindent The pictorial representation of the horizontal composition of the squares above is of the form:

\begin{center}
\begin{tikzpicture}
  \matrix (m) [matrix of math nodes,row sep=4em,column sep=4em,minimum width=2em]
  {
     a&a \\
     a&a \\
     a&a \\
     b&b \\};
  \path[-stealth]
    (m-1-1) edge node [above]{$\alpha\boxminus\alpha'$} (m-1-2)
            edge [blue] node {}(m-2-1)
    (m-2-1) edge [red] node {} (m-2-2)
    (m-1-2) edge [blue] node {}(m-2-2)
    (m-1-1) edge [white] node [black][fill=white] {$\psi\boxminus \psi'$} (m-2-2)
    (m-2-1) edge node [left]{$f$}(m-3-1)
    (m-2-2) edge node [right]{$f$}(m-3-2)
    (m-3-1) edge [red] node {} (m-3-2)
    (m-2-1) edge [white] node [black][fill=white]{$i_f$}(m-3-2)
    (m-3-1) edge [blue]node {}(m-4-1)
    (m-3-2) edge [blue]node {}(m-4-2)
    (m-4-1) edge node [below]{$\beta\boxminus\beta'$} (m-4-2)
    (m-3-1) edge [white] node [black][fill=white]{$\eta\boxminus\eta'$}(m-4-2);
\end{tikzpicture}

\end{center}

\noindent which is clearly a 2-morphism in $V^1_{\gamma C^\Phi}$. This proves that $V^n_{\gamma C^\Phi}=V^1_{\gamma C^\Phi}$ for every positive integer $n$ which proves that the category of morphisms $\gamma C^\Phi_1$ of $C^\Phi$ is equal to $V^1_{\gamma C^\Phi}$. This proves that $\ell\gamma C^\Phi$ and thus that $\ell C^\Phi=1$. This concludes the proof of the proposition.  
\end{proof}

\noindent Observe that the double categories described in theorem \ref{thmmain} are not in general GG. To see this we use the computations provided by lemma \ref{lemmaimportant} \ref{lemmaimportant}.

\begin{ex}\label{exgg}
Let $(\mathcal{B}^*,\mathcal{B})$ be the following decorated bicategory: We make $\mathcal{B}^*$ to be equal to the delooping category $\Omega\mathbb{N}$ of the monoid $\mathbb{N}$. We make $\mathcal{B}$ to be the single object bicategory $2$\textbf{Mat} generated by the usual strictification \textbf{Mat} of the category of complex vector spaces, i.e. \textbf{Mat} is the monoidal category whose collection of objects is $\mathbb{N}$ and whose collection of morphisms is the collection of complex matreces. Composition and tensor product in \textbf{Mat} are the usual composition and tensor product of matreces. Let $\Phi$ be an object of the category $Fun(\Omega\mathbb{N},\mbox{\textbf{Cat}}^\otimes)_{2\mbox{\textbf{Mat}}}$ be such that $\Phi(\ast)=$\textbf{Mat} and such that for every $m\in\mathbb{N}$, $\Phi_m$ acts on \textbf{Mat} by $\Phi_m(n)=mn$ for every $n\in\mathbb{N}$ and by $\Phi_m(A)=A^{\otimes m}$ for every complex matrix $A$. Thus defined $\Phi$ is a functor in $Fun(\Omega \mathbb{N},\mbox{\textbf{Cat}}^\otimes)_{2\mbox{\textbf{Mat}}}$. We show that the double category $C^\Phi$ is not GG. To this end consider the square $(2,id_2)$ in $C^\Phi$. Observe that if $(2,id_2)$ were GG then from lemma \ref{lemmaimportant} $(2,id_2)=id_2^{\otimes 2}$ would admit a factorization as

  \begin{center}

\begin{tikzpicture}
  \matrix (m) [matrix of math nodes,row sep=3em,column sep=7em,minimum width=2em]
  {
     2&4 \\
     0&0 \\};
  \path[-stealth]
    (m-1-1) edge node [left] {$(0,B)$}(m-2-1)
            edge node [above] {$(2,id_2)$}(m-1-2)
    (m-2-1) edge node [below] {$(2,C)$} (m-2-2)
    (m-2-2) edge node [right] {$(0,D)$} (m-1-2);
\end{tikzpicture}
\end{center}

\noindent But were this the case, since the rank $rkB$ and $rkD$ of both $B$ and $D$ is at most 1 then the rank $rk id_2^{\otimes 2}$ of $id_2^{\otimes 2}$ would be at most 1, a contradiction. The 2-morphism $(2,id_2)$ of $C^\Phi$ is thus not GG and $C^\Phi$ is thus not GG.

\end{ex}

\section{Decorations by groups I}

\noindent In this section we study the GG condition on double categories constructed through the methods presented in theorem \ref{thmmain}. We refer the reader to \cite{yo1} for a treatment on GG double categories.

\

\noindent We consider the following situation: Let $G$ be a group. Let $A$ be a commutative monoid. The pair $(\Omega G,2\Omega A)$ is a decorated bicategory. Let $\Phi$ be a functor in $Fun(\Omega G,2\Omega A)$. We consider situations in which the double category $C^\Phi$ associated to $\Phi$ in theorem \ref{thmmain} is GG. Compare this to example \ref{exgg}. We prove the following proposition.

\begin{prop}\label{ggmonoids}
Let $M,A$ be monoids. Suppose $A$ is commutative. Let $\Phi$ be a functor in $Fun(\Omega M,\mbox{\textbf{Cat}}^\otimes)_{2\Omega A}$. Suppose that for every $m\in M$ the endomorphism $\Phi_m$ of $A$ is surjective. In that case $C^\Phi$ is GG.
\end{prop}

\begin{proof}
Let $M,A$ be monoids. Suppose $A$ is commutative. Let $\Phi$ be an object of $Fun(\Omega M,\mbox{\textbf{Cat}}^\otimes)_{2\Omega A}$ such that for every $m\in M$ the endomorphism $\Phi_m$ of $A$ is surjective. We wish to prove in this case that the double category $C^\Phi$ is GG.

We wish to prove that every square in $C^\Phi$ is GG. To do this it is enough to prove that every morphism of $\int_{\Omega M}\Phi$ is GG. Let $(f,\varphi)$ be a morphism in $\int_{\Omega M}\Phi$. We represent $(f,\varphi)$ pictorially as

\begin{center}
\begin{tikzpicture}
  \matrix (m) [matrix of math nodes,row sep=4em,column sep=4em,minimum width=2em]
  {
     \ast&\ast \\
     \ast&\ast \\};
  \path[-stealth]
    (m-1-1) edge [red] node {}(m-1-2)
            edge [white] node [black, fill=white]{$(f,\varphi)$}(m-2-2)
    (m-2-1) edge [red] node {} (m-2-2)
    (m-1-1) edge node [left]{$f$} (m-2-1)
    (m-1-2) edge node [right]{$f$} (m-2-2);
\end{tikzpicture}

\end{center}

\noindent  By lemma \ref{lemmaimportant} and proposition \ref{vertlengthprop} in order to prove that $(f,\varphi)$ is GG it is enough to prove that $(f,\varphi)$ admits a pictorial representaiton as:

\begin{center}
\begin{tikzpicture}
  \matrix (m) [matrix of math nodes,row sep=4em,column sep=4em,minimum width=2em]
  {
     \ast&\ast \\
     \ast&\ast \\
     \ast&\ast \\
     \ast&\ast \\};
  \path[-stealth]
    (m-1-1) edge [red] node {} (m-1-2)
            edge [blue] node {}(m-2-1)
    (m-2-1) edge [red] node {} (m-2-2)
    (m-1-2) edge [blue] node {}(m-2-2)
    (m-1-1) edge[white] node[black][fill=white]{$\psi$}(m-2-2)
    (m-2-1) edge node [left]{$f$}(m-3-1)
    (m-2-2) edge node [right]{$f$}(m-3-2)
    (m-2-1) edge [white] node [black][fill=white]{$(f,\varphi)$}(m-3-2)
    (m-3-1) edge [red] node {} (m-3-2)
    (m-3-1) edge [blue] node {}(m-4-1)
    (m-3-2) edge [blue] node {}(m-4-2)
    (m-4-1) edge [red] node {} (m-4-2)
    (m-3-1) edge[white] node [black][fill=white]{$\eta$}(m-4-2);
\end{tikzpicture}

\end{center}

\noindent for elements $\psi,\eta\in A$. We actually prove that $(f,\varphi)$ admits a pictorial representation as:

\begin{center}
\begin{tikzpicture}
  \matrix (m) [matrix of math nodes,row sep=4em,column sep=4em,minimum width=2em]
  {
     \ast&\ast \\
     \ast&\ast \\
     \ast&\ast \\};
  \path[-stealth]
    (m-1-1) edge [red] node {} (m-1-2)
            edge [blue] node {}(m-2-1)
    (m-2-1) edge [red] node {} (m-2-2)
    (m-1-2) edge [blue] node {}(m-2-2)
    (m-1-1) edge [white] node [black][fill=white]{$\psi$}(m-2-2)
    (m-2-1) edge node [left]{$f$}(m-3-1)
    (m-2-2) edge node [right]{$f$}(m-3-2)
    (m-2-1) edge [white] node [black][fill=white]{$(f,\varphi)$}(m-3-2)
    (m-3-1) edge [red] node {} (m-3-2);
\end{tikzpicture}

\end{center}

\noindent for some element $\psi\in M$. The above pictorial equation is equivalent to the equation $(f,\varphi)=(f,1_A)(1_M,\psi)$ which in turn is equivalent to the equation $(f,\varphi)=(f,\Phi_f(\psi))$. A solution to this equation is thus found by solving the equation $\varphi=\Phi_f(\psi)$ in $A$, but the condition that $\Phi_f$ is surective guarantees the existence of such a solution. By lemma \ref{lemmaimportant} we thus have that every morphism in $\int_{\Omega M}\Phi$ is a morphism in $V^1_{\gamma C^\Phi}$. We conclude that $C^\Phi$ is GG.
\end{proof}

\begin{remark}
The arguments employed in the proof of proposition \ref{ggmonoids} work in the more general setting in which instead of considering bicategories of the form $2\Omega A$ we consider bicategories of the form $2C$ for a monoidal category $C$ as long as the category $C$ satisfies the condition that every morphism in $C$ factors through the tensor unit 1$_C$ of $C$. We avoid proving proposition \ref{ggmonoids} in this generality for simplity.
\end{remark}

\begin{cor}\label{lemma1rigid}

Let $G$ be a group. Let $A$ be a commutative monoid. Let $\Phi$ be a functor in $Fun(\Omega G,\mbox{\textbf{Cat}}^\otimes)_{\Omega A}$. In that case the double category $C^\Phi$ is GG.

\end{cor}

\section{Decorations by groups II}

\noindent In this final section we investigate a construction inverse to that presented in theorem \ref{thmmain} in the case of bicategories decorated by groups. The results of this section should be regarded as an extension of corollary \ref{1stcorollary}.

\

\noindent We consider the following problem: Let $G$ be a group. Let $A$ be a commutative monoid. Let $C$ be a double category satisfying the conditions of proposition \ref{ggmonoids}. We wish to construct a functor $\Phi$ in $Fun(\Omega G,\mbox{\textbf{Cat}}^\otimes)_{2\Omega A}$ such that $C$ and $C^\Phi$ are somehow related. We begin ou investigation of this problem by proving the following proposition.

\begin{prop}\label{groups1}
Let $G$ be a group. Let $A$ be a commutative monoid. Let $C$ be an object in \textbf{dCat}$_{(\Omega G,2\Omega A)}$ such that $C$ is GG and such that $\ell C=1$. In that case there exists a functor $\Phi$ in $Fun(\Omega G,\mbox{\textbf{Cat}}^\otimes)_{2\Omega A}$ and a double functor $\pi^C:C^\Phi\to C$ such that $H^*\pi^C=id_{(\Omega G,2\Omega A)}$ and $\pi^C_1$ is full. 
\end{prop}

\begin{proof}
Let $G$ be a group. Let $A$ be a commutative monoid. Let $C$ be a GG double internalization of $(\Omega G,2\Omega A)$ such that $\ell C=1$. We wish to prove that there exist a functor $\Phi$ in $Fun(\Omega G,\mbox{\textbf{Cat}}^\otimes)_{2\Omega A}$ and a full double functor $\pi^C:C^\Phi\to C$ such that $H^*\pi^C=id_{(\Omega G,2\Omega A)}$.

In order to obtain a functor $\Phi$ in $Fun(\Omega G,\mbox{\textbf{Cat}}^\otimes)_{\Omega A}$ it is enough to define a representation, which we will keep denoting by $\Phi$, of $G$ in $End(A)$. Given $g\in G$ and $a\in A$ we make $\Phi_g(a)$ to be defined through the following diagram in $C^\Phi$:

\begin{center}
\begin{tikzpicture}
  \matrix (m) [matrix of math nodes,row sep=4em,column sep=4em,minimum width=2em]
  {
     \ast&\ast \\
     \ast&\ast \\
     \ast&\ast \\
     \ast&\ast \\};
  \path[-stealth]
    (m-1-1) edge [red] node {} (m-1-2)
            edge node [left]{$f$}(m-2-1)
    (m-2-1) edge [red] node {} (m-2-2)
    (m-1-2) edge node [right]{$f$}(m-2-2)
    (m-1-1) edge [white] node [black][fill=white]{$i_f$}(m-2-2)
    (m-2-1) edge [blue]node {}(m-3-1)
    (m-2-2) edge [blue]node {}(m-3-2)
    (m-2-1) edge [white] node [black][fill=white]{$a$}(m-3-2)
    (m-3-1) edge [red] node {} (m-3-2)
    (m-3-1) edge node [left]{$f^{-1}$}(m-4-1)
    (m-3-2) edge node [right]{$f^{-1}$}(m-4-2)
    (m-4-1) edge [red] node {} (m-4-2)
    (m-3-1) edge [white] node [black][fill=white]{$i_{f^{-1}}$}(m-4-2);
\end{tikzpicture}

\end{center}

\noindent To see that this defines a representation of $G$ in $End(A)$ we first observe that by the equation $HC=\Omega A$ it follows that the above 2-morphism of $C$ is globular and thus is an element of $A$. The fact that as defined above the function $\Phi$ defines a morphism from $G$ to $End(A)$ follows immediatly from the structure equations defining $C$. We thus obtain a representation $\Phi:G\to End(A)$ which forms the morphism function of a functor $\Phi$ in $Fun(\Omega G,\mbox{\textbf{Cat}}^\otimes)_{2\Omega A}$. 

We now construct a full double functor $\pi^C:C^\Phi\to C$ satisfying the equation $H^*\pi^C=id_{(\Omega G,2\Omega A)}$. By the fact that $C$ is a length 1 GG internalization of $(\Omega G,2\Omega A)$ and by Lemma 4.11 of \cite{yo2} we obtain a full double functor 

\[\pi^{C,1}:V_1^{(\Omega G,2\Omega A)}\to C_1\]

\noindent such that the restriction of $\pi^{C,1}$ to the collections of 1- and 2-cells $2\Omega A_1$ and $2\Omega A_2$ of $2\Omega A$ is equal to $id_{2\Omega A_1}$ and $id_{2\Omega A_2}$ respectively. An easy computation proves that $V_1^{(\Omega G,2\Omega A)}$ is equal to the delooping category $\Omega(G\ast A)$ of the free product $G\ast A$ of $G$ and $A$. From this and from the fact that for every $g\in G$ and $a\in A$ the 2-morphisms $i_a$ and $b$ satisfy the relation 

\[i_{g}^{-1}ai_g=\Phi_g(a)\]

\noindent it follows that there exists an epic morphism $\pi^C_1:G\rtimes_\Phi A\to End_{C_1}(id_\ast)$ making the following triangle commute:

 \begin{center}

\begin{tikzpicture}
  \matrix (m) [matrix of math nodes,row sep=2em,column sep=2em,minimum width=2em]
  {
     G\ast A&&End_{C_1}(id_\ast) \\
     &G\rtimes_\Phi A& \\};
  \path[-stealth]
    (m-1-1) edge node {}(m-2-2)
            edge node [above] {$\pi^{C}$}(m-1-3)
    (m-2-2) edge node [below] {$\pi^{1,C}$} (m-1-3);
\end{tikzpicture}
\end{center}

\noindent By the way $\pi^{1,C}$ is defined it follows that the pair $\pi^C=(id_G,\pi^1)$ is a full double functor from $C^\Phi$ to $C$ satisfyng the equation $H^*\pi^C=id_{(\Omega G,2\Omega A)}$. This concludes the proof. 

\end{proof}

\noindent Let $(\mathcal{B}^*,\mathcal{B})$ be a decorated bicategory. We will write $\mbox{\textbf{gCat}}^{\ell=1}_{(\mathcal{B}^*,\mathcal{B})}$ for the category whose objects are GG double categories $C$ satisfying the equations $H^*C=(\mathcal{B}^*,\mathcal{B})$ and $\ell C=1$. Given objects $C,C'$ in $\mbox{\textbf{gCat}}^{\ell=1}_{(\mathcal{B}^*,\mathcal{B})}$ we make the collection of morphisms from $C$ to $C'$ to be the collection of double functors $F$ from $C$ to $C'$ such that the object functor $F_0$ of $F$ is equal to the identity endofunctor of $\mathcal{B}^*$. Observe that as defined $\mbox{\textbf{gCat}}^{\ell=1}_{(\mathcal{B}^*,\mathcal{B})}$ is not a subcategory of \textbf{dCat}$_{(\mathcal{B^*,\mathcal{B}})}$. We extend proposition \ref{groups1} through the following lemma.

\begin{lem}\label{groups2}
Let $G$ be a group. Let $A$ be a commuatitve monoid. The function associating to every $C$ in \textbf{gCat}$_{(\Omega G,2\Omega A)}^{\ell=1}$ the functor $\Phi^C$ extends to a functor

\[\Phi^\bullet:\mbox{\textbf{gCat}}_{(\Omega G,2\Omega A)}^{\ell=1}\to Fun(\Omega G,\mbox{\textbf{Cat}}^\otimes)_{2\Omega A}\]

\end{lem}

\begin{proof}
Let $G$ be a group. Let $A$ be a commutative monoid. We wish to prove that the operation of associating the functor $\Phi^C$ to an object $C$ in \textbf{gCat}$_{(\Omega G,2\Omega A)}^{\ell=1}$ extends to a functor $\Phi^\bullet:\mbox{\textbf{gCat}}_{(\Omega G,2\Omega A)}^{\ell=1}\to Fun(\Omega G,\mbox{\textbf{Cat}}^\otimes)_{2\Omega A}$. 

Let $C,C'$ be objects in \textbf{gCat}$_{(\Omega G,2\Omega A)}^{\ell=1}$. Let $g\in G$. We will write $i_g$ and $i'_g$ for the horizontal identities of $g$ in $C$ and in $C'$ respectively. The horizontal identities $i_g,i'_g$ are related by the equation $F_1(i_g)=i'_g$. From this equation we obtain:

\[F_1(i_{g^{-1}}ai_g)=i'^{-1}_gF_1(a) i'_g\]

\noindent for every $a\in A$. The equation:

\[F_1\Phi^C_g(a)=\Phi^{C'}_g(F_1(a))\]

\noindent follows for every $g\in G$ and for every $a\in A$. Observe that the fact that $F$ is a double functor implies that $F_1\restriction_A$ is an endomorphism of $A$. From this and from the equations above it is immediate that the following square commutes for every $g\in G$:

 \begin{center}

\begin{tikzpicture}
  \matrix (m) [matrix of math nodes,row sep=3em,column sep=7em,minimum width=2em]
  {
     \Omega A&\Omega A \\
     \Omega  A&\Omega A \\};
  \path[-stealth]
    (m-1-1) edge node [left] {$\Phi_g^C$}(m-2-1)
            edge node [above] {$F_1\restriction_A$}(m-1-2)
    (m-2-1) edge node [below] {$F_1\restriction_A$} (m-2-2)
    (m-2-2) edge node [right] {$\Phi_g^{C'}$} (m-1-2);
\end{tikzpicture}
\end{center}

\noindent This proves that $F_1\restriction_A$ defines a natural transformation from $\Phi^C$ to $\Phi^{C'}$. It is immediate that the function associating to every morphism $F$ in \textbf{gCat}$_{(\Omega G,2\Omega A)}^{\ell=1}$ the natural transformation $F_1\restriction_A$ extends the function associating to every double category $C$ in \textbf{gCat}$_{(\Omega G,2\Omega A)}^{\ell=1}$ to a functor from \textbf{gCat}$_{(\Omega G,2\Omega A)}^{\ell=1}$ to $ Fun(\Omega G,\mbox{\textbf{Cat}}^\otimes)_{2\Omega A}$. This concludes the proof of the lemma.

\end{proof}

\begin{prop}\label{groups3}
Let $G$ be a group. Let $A$ be a commutative monoid. In that case the following relation holds:

\[\Phi^\bullet\vdash C^\bullet\]
\end{prop}

\begin{proof}
Let $G$ be a group. Let $A$ be a commutative monoid. We wish to prove in this case that $\Phi^\bullet\vdash C^\bullet$.

We provide a counit-unit pair $(\epsilon,\eta)$ for the adjunction $\Phi^\bullet\vdash C^\bullet$. We begin by defining $\epsilon$. Let $C$ be an object in \textbf{gCat}$_{(\Omega G,2\Omega A)}^{\ell=1}$. We make $\epsilon_C$ to be the functor $\pi^C$ deifned in lemma \ref{groups2}. We prove that the collection $\epsilon$ of double functors $\pi^C$ with $C$ running through \textbf{gCat}$_{(\Omega G,2\Omega A)}^{\ell=1}$ is a natural transformation from $C^\bullet\Phi^\bullet$ to the identity endofunctor of \textbf{gCat}$_{(\Omega G,2\Omega A)}^{\ell=1}$. To see this let $C,C'$ be objects in \textbf{gCat}$_{(\Omega G,2\Omega A)}^{\ell=1}$ and let $F:C\to C'$ be a double functor such that $F_0=id_{\Omega G}$. In that case the restriction $F_1\restriction_A$ of the morphism functor $F_1$ of $F$ is a natural transformation from $\Phi^C$ to $\Phi^{C'}$ by lemma \ref{groups2}. From this and from the equation $H^*\pi^C=id_{(\Omega G,2\Omega A)}$ it follows that the following diagram is commutative:

\begin{center}

\begin{tikzpicture}
  \matrix (m) [matrix of math nodes,row sep=3em,column sep=7em,minimum width=2em]
  {
     C^{\Phi^C}&C'^{\Phi^{C'}} \\
     C &C' \\};
  \path[-stealth]
    (m-1-1) edge node [left] {$\pi^C$}(m-2-1)
            edge node [above] {$F$}(m-1-2)
    (m-2-1) edge node [below] {$F$} (m-2-2)
    (m-1-2) edge node [right] {$\pi^{C'}$} (m-2-2);
\end{tikzpicture}
\end{center}

\noindent from which we conclude that $\epsilon$ is indeed a natural transformation from $C^\bullet\Phi^\bullet$ to the identity endofunctor of \textbf{gCat}$_{(\Omega G,2\Omega A)}^{\ell=1}$.

We now define $\eta$. Observe that for every $\Phi$ in $Fun(\Omega G,\mbox{\textbf{Cat}}^\otimes)_{(\Omega G,2\Omega A)}$ the monoid of endomorphisms of the horizontal identity $i_\ast$ of the only object in $C^\Phi$ is equal to $G\rtimes_\Phi A$ and thus $\pi^{C^\Phi}$ is equal to $\Phi$. We make $\eta$ to be the collection of identity natural transformations $id_\Phi$ with $\Phi$ running through the collection of all functors in $Fun(\Omega G,\mbox{\textbf{Cat}}^\otimes)_{(\Omega G,2\Omega A)}$.

We prove that the pair $(\epsilon,\eta)$ satisfies the triangle identities. That is, we prove that the following triangles commute:

\begin{center}

\begin{tikzpicture}
  \matrix (m) [matrix of math nodes,row sep=3em,column sep=5em,minimum width=2em]
  {
     C^\bullet&C^\bullet &\Phi^\bullet&\Phi^\bullet\\
     \Phi^\bullet C^\bullet\Phi^\bullet&&C^\bullet\Phi^\bullet C^\bullet & \\};
  \path[-stealth]
    (m-1-3) edge node [left] {$\Phi^\bullet\eta$}(m-2-3)
            edge node [above]{$id_{\Phi^\bullet}$}(m-1-4)
    (m-2-3) edge node [right] {$\epsilon\Phi^\bullet$} (m-1-4)
   (m-1-1) edge node [left] {$\eta C^\bullet$}(m-2-1)
            edge node [above]{$id_{C^\bullet}$}(m-1-2)
    (m-2-1) edge node [right] {$C^\bullet\epsilon$} (m-1-2) ;
\end{tikzpicture}
\end{center}

\noindent The commutativity of the triangle on the left hand side reduces, in our setting, to proving that for every functor $\Phi$ in $Fun(\Omega G,\mbox{\textbf{Cat}}^\otimes)_{(\Omega G,2\Omega A)}$ the equation $\pi^{C^\Phi}=id_{C^\Phi}$ holds. To prove this observe that the monoid of endomorphisms $End_{C^\Phi}(i_\ast)$ of the horizontal identity $i_\ast$ of the only object $\ast$ in $C^\Phi$ is equal to $G\rtimes_\Phi A$. From this and from the way $\pi^C$ was defined we conclude that $\pi^\Phi=id_{C^\Phi}$. The commutativity of the second triangle above reduces to proving that the equation $id_{\Phi^C}=\Phi^{\pi^C}$ holds for every object $C$ in \textbf{gCat}$_{(\Omega G,2\Omega A)}^{\ell=1}$. This follows from arguments analogous as the ones described in proving the commutativity of the triangle on the left hand side above. This concludes the proof.

\end{proof}

\bibliographystyle{plain}
\bibliography{biblio}

\end{document}